\documentclass[leqno,12pt]{article}
\usepackage{latexsym}
\usepackage{amsmath}
\usepackage{amssymb}
\usepackage{enumerate}
\usepackage{graphicx}

\usepackage{theorem}
\newtheorem{theorem}{Theorem}[section]

\newtheorem{lemma}[theorem]{Lemma}
\newtheorem{corollary}[theorem]{Corollary}

\newtheorem{TA}{Differentiable Exotic Sphere Theorem I.}
\newtheorem{TAII}{Differentiable Exotic Sphere Theorem II.}
\newtheorem{TB}{Differentiable Twisted Sphere Theorem.}

\theorembodyfont{\rmfamily}
\newtheorem{proof}{\textmd{\textit{Proof.}}}

\newtheorem{remark}[theorem]{Remark}
\newtheorem{example}[theorem]{Example}
\newtheorem{definition}[theorem]{Definition}

\newtheorem{acknowledgement}{\textmd{\textit{Acknowledgements.}}}

\makeatletter

\@addtoreset{equation}{section}
\makeatother

\newcommand{\qedd}{\hfill \Box}
\newcommand{\ve}{\varepsilon}
\newcommand{\lra}{\longrightarrow}
\newcommand{\ora}{\overrightarrow}

\newcommand{\wt}{\widetilde}

\newcommand{\ol}{\overline}
\newcommand{\B}{\ensuremath{\mathbb{B}}}

\newcommand{\R}{\ensuremath{\mathbb{R}}}

\newcommand{\Sph}{\ensuremath{\mathbb{S}}}

\newcommand{\cN}{\ensuremath{\mathcal{N}}}

\newcommand{\cP}{\ensuremath{\mathcal{P}}}
\newcommand{\cU}{\ensuremath{\mathcal{U}}}

\def\diam{\mathop{\mathrm{diam}}\nolimits}

\def\vol{\mathop{\mathrm{vol}}\nolimits}

\def\supp{\mathop{\mathrm{supp}}\nolimits}
\def\id{\mathop{\mathrm{id}}\nolimits}
\def\Ric{\mathop{\mathrm{Ric}}\nolimits}
\def\Cut{\mathop{\mathrm{Cut}}\nolimits}
\def\bL{\mathop{\mathrm{Lip^b}}\nolimits}
\def\L{\mathop{\mathrm{Lip}}\nolimits}
\def\Crit{\mathop{\mathrm{Crit}}\nolimits}
\def\Conv{\mathop{\mathrm{Conv}}\nolimits}

\title{Approximations of Lipschitz maps via immersions\\
and differentiable exotic sphere theorems
}

\author{Kei KONDO\footnote{Partly supported by the two 
Grant-in-Aids for Science Reserch (C), 
JSPS KAKENHI Grant Numbers 15K04846 and 16K05133.} \ $\cdot$ \ Minoru TANAKA}
\date{\today}
\begin{document}
\maketitle

\begin{abstract}
As our main theorem, we prove that a Lipschitz map 
from a compact Riemannian manifold $M$ 
into a Riemannian manifold $N$ 
admits a smooth approximation via immersions if the map has 
no singular points on $M$ in the sense of F.H.\,Clarke, 
where $\dim M \le \dim N$. As its corollary, 
we have that if a bi-Lipschitz homeomorphism 
between compact manifolds and its inverse map have no singular points in the same sense, then they are diffeomorphic. We have three applications 
of the main theorem: The first two of them are two differentiable sphere theorems for a pair of topological spheres including that of exotic ones. 
The third one is that a compact $n$-manifold $M$ is a twisted sphere 
and there exists a bi-Lipschitz homeomorphism 
between $M$ and the unit $n$-sphere $S^n(1)$ 
which is a diffeomorphism except for a single point, 
if $M$ satisfies certain two conditions with respect 
to critical points of its distance function 
in the Clarke sense. Moreover, we have three corollaries 
from the third theorem; 
the first one is that for any twisted sphere $\Sigma^n$ 
of general dimension $n$, there exists a bi-Lipschitz homeomorphism between $\Sigma^n$ and $S^n(1)$ 
which is a diffeomorphism except for a single point. 
In particular, there exists such a map 
between an exotic $n$-sphere $\Sigma^n$ 
of dimension $n>4$ and $S^n(1)$; 
the second one is that if an exotic $4$-sphere 
$\Sigma^4$ exists, 
then $\Sigma^4$ does not satisfy one 
of the two conditions above; 
the third one is that for any Grove-Shiohama type $n$-sphere $N$, there exists a bi-Lipschitz homeomorphism between $N$ and $S^n(1)$ which is a diffeomorphism except for one of points that attain their diameters.
\end{abstract}

\section{Introduction}\label{sec1}

\subsection{Motivations and our main theorem}\label{2016_01_10_motiv}

We first mention our motivations behind our purpose, which are 
exotic structures: No one needs the introduction of exotic $n$-dimensional spheres 
$\Sigma^n$, however $\Sigma^n$ were very first discovered by Milnor \cite{M} in the case of $n=7$, which is, by definition, homeomorphic to the standard $n$-sphere $S^n$ but not diffeomorphic to it. Note that  
the smooth $4$-dimensional Poincar\'e conjecture (SPC$4$), that is, 
the problem of the existence of an exotic structure on the $4$-sphere, is still open.\par   
Due to exotic structures, we always have technical difficulties 
when investigating whether a topological sphere theorem can 
be reinforced into a differentiable sphere theorem. 
The difficulties become clearer from global Riemannian geometry's 
stand point: It follows from Smale's h-cobordism theorem \cite{Sm2} together with \cite{Sm1} that every homotopy sphere of dimension 
$n\ge5$ is a twisted sphere, which is a smooth manifold obtained 
by glueing two standard $n$-discs along their 
boundaries under a boundary diffeomorphism. 
This implies that every 
$\Sigma^n$ ($n>4$) is actually twisted, since $\Sigma^n$ 
is a homotopy sphere. Applying Weinstein's deformation technique 
for metrics (\cite[Proposition C]{W}) to both of two discs embedded smoothly into a twisted sphere $X$ of general dimension, 
we see that $X$ admits a metric such that 
the cut locus\footnote{The cut locus $\Cut (p)$ of a point $p$ 
in a complete Riemannian manifold $M$ is, by definition, 
the closure of the set of all points $x \in M$ 
such that there are at least two minimal geodesics emanating from $p$ to $x$. Then, a point in $\Cut (p)$ is called the cut point of $p$. 
For example, $\Cut (N)$ of $N:=(0,\ldots,0,1)\in S^n(1)$ is 
$\{(0,\ldots,0,-1)\}$, where $S^n(1):=\{v \in \R^{n+1}\,|\,\|v\|=1\}$. 
Note that the distance function from a point $x$ of $M$ is 
not differentiable at the cut point of $x$.} of some point 
on $X$ is a single point. By all together above, we have 

\begin{theorem}\label{2015_07_27_weinstein}
\label{2015_03_31_thm1.2}{\rm (Also see \cite[Proposition 7.19]{B})} Every exotic sphere $\Sigma^n$ 
of dimension $n>4$ admits a metric such that 
there is a point whose cut locus consists 
of a single point.
\end{theorem}

Thus, it is very difficult for us to notice the difference between 
$\Sigma^n$ and $S^n$ from the point of view 
of two exponential maps at the single point on $\Sigma^n$ 
and at any point on $S^n$. 
For example, one of open problems in global Riemannian 
geometry is if a Grove-Shiohama type $n$-sphere can be 
diffeomorphic to $S^n$. Here, a complete Riemannian manifold $V$ 
is called a {\em Grove-Shiohama type sphere} if 
sectional curvature $K_V \ge 1$ and its diameter $\diam (V) > \pi/2$ (\cite{GS}). 
Since such a $V$ is twisted, 
from Theorem \ref{2015_07_27_weinstein} 
we can infer that some single cut point on $V$ 
is a big obstacle whenever 
approximating a homeomorphism, 
in fact which is bi-Lipschitz, between $V$ and $S^n$ 
by diffeomorphisms.

\bigskip

Hence, there is a cardinal importance to do an analysis 
of such singular points on an arbitrary manifold. For this, 
we employ a notion used in non-smooth analysis 
of F.H. Clarke (\cite{C1}, \cite{C2}), i.e., 
a non-singular point for a Lipschitz function/map in this article. 
The following example shows that the non-smooth analysis 
is a strong tool in differential geometry.  

\begin{example}\label{2016_03_28_exa_1}
Let $M$ be a complete Riemannian manifold, 
$d$ the distance function of $M$. Take any point $p \in M$, and fix it. 
Set $d_p (x):= d(p, x)$ for all $x \in M$. 
Then, the point $q \in M \setminus \{ p \}$ 
is a {\em critical point of $d_p$} 
(or {\em critical point for $p$}) in the sense 
of Grove-Shiohama \cite{GS}, 
if for every nonzero tangent vector $v \in T_qM$ at $q$, 
there exists a minimal geodesic segment $\gamma$ emanating from $q$ to $p$ such that  
\[
\angle(v,\dot{\gamma}(0)) \le \frac{\pi}{2}.
\] 
Here, $\angle(v,\dot{\gamma}(0))$ denotes 
the angle between $v$ and $\dot{\gamma}(0):=(d\gamma/dt)(0)$. Note that 
a critical point of $d_p$ is the cut point of $p$. 
Assume that, for some $r>0$, 
$\partial B_r (p):=\{x\in M\,|\,d(p,x)=r\}$ has no 
critical points of $d_p$. 
By Gromov's isotopy lemma \cite{G1}, 
$\partial B_r (p)$ is a topological submanifold of $M$. 
Since $\partial B_r (p)$ is also free of critical points of $d_p$ 
in the sense of Clarke (See Example \ref{2016_03_25_exa_GS}), 
it follows from Clarke's implicit function theorem \cite{C4} that 
\begin{center}
in fact, $\partial B_r (p)$ is a {\em Lipschitz} submanifold of $M$.
\end{center}  
We are going to give the definition of the Clarke sense later (See  
Definition \ref{2014_02_20_def_1} in Sect\,\ref{2016_01_24_sub1.3}). 
\end{example}

Our purpose of this article is 
{\em to establish an approximation method for a Lipschitz map via diffeomorphisms using the notion of non-smooth analysis 
and to apply this method to prove differentiable sphere theorems}. 
That is, our main theorem is as follows: 

\begin{theorem}\label{2015_07_30_theorem2}
{\rm (Main Theorem)} 
Let $F: M \lra N$ be a Lipschitz map from a compact Riemannian manifold 
$M$ into a Riemannian manifold $N$, where $\dim M \le \dim N$. 
If $F$ has no singular points on $M$ (in the sense of Clarke), 
then for any $\eta>0$, there exists a smooth immersion 
$f_\eta$ from $M$ into $N$ such that   
\[
\max_{x\in M}d_N(F(x), f_\eta(x))<\eta, \quad 
\L(f_\eta) \le \L(F)(1+\eta).
\]
Here, $\L(f_\eta)$ and $\L (F)$ denote the Lipschitz constants 
of $f_\eta$ and $F$, respectively, i.e.,  
\[
\L(f_\eta):=\sup\left\{\frac{d_N(f_\eta(x),f_\eta(y))}{d_M(x,y)} \, \bigg| \,x,y\in M, x\ne y\right\}
\]
and 
\[
\L(F):=\sup\left\{\frac{d_N(F(x),F(y))}{d_M(x,y)} \, \bigg| \,x,y\in M, x\ne y\right\},
\]
where $d_M$ and $d_N$ are the distance functions of $M$ and $N$, respectively.
\end{theorem}

\begin{remark}\label{2016_03_11_rem}
We give three remarks on Theorem \ref{2015_07_30_theorem2}. 
\begin{itemize} 
\item 
The definition of a {\it singular point} for a Lipschitz map, 
i.e., non-smooth analysis of Clarke, 
is given in Sect\,\ref{2016_01_24_sub1.3}.
\item
It should be emphasized that the reason why Clarke introduced 
the notion of non-smooth analysis was not for approximations 
of a Lipschitz map via diffeomorphisms, 
but for the inverse function theorem for a Lipschitz map, which contains the classical one as a special case. See also Example \ref{2015_07_25_Exa_3} below. 
\item Our approximation method for Lipschitz maps 
to prove Theorem \ref{2015_07_30_theorem2} generalizes 
the whole Grove-Shiohama one in \cite{GS}. See Sect.\,\ref{sect2}.
\end{itemize}
\end{remark}

As an indirect corollary of Theorem \ref{2015_07_30_theorem2}, 
for any twisted sphere $\Sigma^n$ of any dimension $n$, 
we can construct a concrete bi-Lipschitz 
homeomorphism $F$ between $\Sigma^n$ and $S^n$ 
which is a diffeomorphism except for a single point 
(See Corollary \ref{2016_01_10_cor3}). 
Therefore, by this result together with 
Theorem \ref{2015_07_27_weinstein}, the existence of exotic $n$-spheres ($n>4$) implies 
that we cannot approximate $F$ by diffeomorphisms. 
Hence, we must give careful consideration to sufficient conditions 
for a pair of topological spheres admitting a single cut point 
of some point, hence even for that of exotic ones, to be diffeomorphic.\par 
As applications of Theorem \ref{2015_07_30_theorem2}, 
we prove differentiable sphere theorems not only for such a pair above, 
but for a pair of exotic $n$-spheres ($n>4$), where we give the 
sufficient conditions for each pair to be diffeomorphic. See 
Sect.\,\ref{2016_01_11_new}.

\subsection{Non-smooth analysis and a corollary 
of Theorem \ref{2015_07_30_theorem2}}\label{2016_01_24_sub1.3}

F.H.\,Clarke very first established the non-smooth analysis 
in \cite{C1} and \cite{C2}. 
It is a strong tool not only in the optimal control theory (cf.\,\cite{C3}), 
but also in differential geometry (e.g.,\,
\cite{JMS}, \cite{Ri}, and Example \ref{2016_03_28_exa_1}): 
Let $M$, $N$ be Riemannian manifolds, 
and let $F: M\lra N$ be a Lipschitz map. 
In the case of $N=\R$, let $f:=F:M\lra\R$. 
By Rademacher's theorem \cite{R}, there exists a set $E_F\subset M$ of measure zero such that the differential $dF$ of $F$ exists 
on $M\setminus E_F$. Then, for each point $x$, 
there exists a sequence $\{x_i\}$ of $x_i \in M\setminus E_F$ 
convergent to $x$, and hence we can define 
the {\em generalized differential $\partial F(x)$ of $F$ at $x\in M$} 
as follows:
\begin{equation}\label{2015_08_02_Clarke_1}
\partial F (x):= \Conv (\{\lim_{i \to \infty}
dF_{x_i}\,|\,dF_{x_i} \ \text{exists as} \ x_i\in M\setminus E_F \lra x\}),
\end{equation}
where ``$\Conv (\,\cdot\,)$" means ``convex hull''. 
In the case of $N = \R$, we call $\partial f(x)$ 
the {\em generalized gradient of $f$ at $x$}, i.e., 
\[
\partial f (x)
:= \Conv (\{\lim_{i\to\infty} \nabla f (x_i)\,|\, \nabla f (x_i) \ \text{exists as} \ 
x_i\in M\setminus E_f \lra x\}).
\]
Here $\nabla f$ denotes the gradient vector field of $f$. 
Note that both definitions above do not depend 
on atlases (e.g.,\,\cite{JMS}), 
and that $\partial F (x)$, $\partial f (x)$ are 
compact convex sets. 

\begin{definition}{\rm (\cite{C1}, \cite{C2})}\label{2014_02_20_def_1} 
Let $M$, $N$ be the same as above, and 
$U \subset M$ an open set.
\begin{itemize}
\item A point $x \in U$ is said to be {\it non-singular for a Lipschitz map $F: U\lra N$}, if every element in $\partial F(x)$ is of maximal rank. 
\item A point $x\in U$ is said to be {\it non-critical for a Lipschitz function 
$f : U \lra \R$}, if 
\[
o \not\in \partial f (x),
\]
where $o$ denotes the zero tangent vector at $x$.
\end{itemize}
\end{definition}

We give several examples of Definition \ref{2014_02_20_def_1} 
with a few related remarks. 

\begin{example}\label{2015_07_25_Exa_1}
Consider two functions $f_1 (x):= x^2$, $f_2 (x):= x+2$ on $(-2,3)$. 
Define the Lipschitz function $f(x):= \max\{f_1 (x), f_2(x)\}$ on $(-2,3)$, i.e.,
\[
f(x) = 
\begin{cases}
\ x^2  \ &\text{on} \ (-2, -1)\cup(2,3),\\
\ x+2  \ &\text{on} \ [-1,2].
\end{cases}
\] 
Note that $f$ is not differentiable at $x = -1, 2$. Since 
\[
\lambda \frac{df_1}{dx}(-1) + (1-\lambda) \frac{df_2}{dx}(-1)= -3\lambda +1
\] 
for all $\lambda \in [0,1]$, we have $\partial f(-1) = [-2, 1]$. As well as above, 
we have $\partial f(2) = [1, 4]$. Since $0 \in \partial f(-1)$ and $0 \not\in \partial f(2)$, $x=-1$ is a critical point of $f$ and $x=2$ is not that of $f$. 
Note that $f(-1) = 1$ is the minimum value of $f$.
\end{example}

\begin{example}(\cite[Remark 1]{C2})\label{2015_07_25_Exa_2} 
Let $F:\R^2\lra\R^2$ be the Lipschitz map defined 
by $F(x, y):= (|x|+y, 2x +|y|)$. 
Note that $F$ is not differentiable at $(x,y)=(0,0)$. Then, 
\begin{align*}
\partial F(0, 0)
&=\Conv\left(\left\{ 
\left(
\begin{array}{cc}
1 & 1\\
2 & 1
\end{array}
\right),
\left(
\begin{array}{cc}
1 & 1\\
2 & -1
\end{array}
\right),
\left(
\begin{array}{cc}
-1 & 1\\
2 & 1
\end{array}
\right),
\left(
\begin{array}{cc}
-1 & 1\\
2 & -1
\end{array}
\right)
\right\}\right)\\[2mm]
&= \left\{ 
\left(
\begin{array}{cc}
s & 1\\
2 & t
\end{array}
\right)\,\bigg|\ |s| \le 1, \ |t|\le1
\right\}.
\end{align*}
Thus, $(0,0)$ is non-singular for $F$.
\end{example}

\begin{example}(\cite[Lemmas 3,\,4]{C2})\label{2015_07_25_Exa_3} 
Let $F:\R^n\lra\R^k$ 
be a map that is Lipschitz near some point $p \in\R^n$, where $n\le k$. 
Assume that $p$ is non-singular for $F$. Then for any $u\in S^{n-1}(1)$, 
$\partial F(p)u:=\{Au\,|\,A \in \partial F(p)\}$ is convex in $\R^k$ and 
$o\not\in\partial F(p)u$. Hence, there exist a vector $v \in S^{k-1}(1)$ and 
a number $\delta (p)>0$ such that 
\[
\langle Au, v \rangle \ge 2\delta (p)
\]
for all $A \in \partial F(p)$. Moreover, by the definition of $\partial F(p)$, 
we may find a number $r(p)>0$ satisfying 
\[
\langle Au, v \rangle \ge \delta (p)
\]
for all $A \in \partial F(q)$ and all 
$q \in B_{r(p)}(p):=\{x \in \R^n\,|\, \|p-x\|< r(p)\}$. 
Furthermore, one can prove that if $q_1, q_2 \in \ol{B_{r(p)}(p)}:=\{x \in \R^n\,|\, \|p-x\|\le r(p)\}$, 
$F$ is $\delta (p)$-expanding near $p$, that means that 
\[
\|F(q_2)-F(q_1)\| \ge \delta (p)\|q_2-q_1\|.
\]
Using this inequality, Clarke proved  the existence of a local Lipschitzian inverse at a non-singular point of a Lipschitz map (See \cite[Theorem 1]{C2}).
\end{example}

\begin{example}\label{2016_03_25_exa_GS}
(Critical points of Grove-Shiohama \cite{GS})\label{2015_07_25_Exa_4} Let $M$ be a complete Riemannian manifold, 
$d$ the distance function of $M$, and 
fix $p\in M$. Assume that $q \in M$ is not a critical point for $p$ 
in the sense of Grove-Shiohama (Refer to Example \ref{2016_03_28_exa_1} for its definition). 
By definition, there exists $w \in T_q M \setminus \{o\}$ 
such that $\angle(w, \dot{\gamma}(0)) > \pi/2$ holds 
for all minimal geodesic segments 
$\gamma$ emanating from $q$ to $p$. 
Hence, $o \not\in \partial d_p (q)$, where 
$d_p(x):=d(p,x)$ for all $x \in M$.  
Note that a point $q\ne p$ is a critical point of $d_p$ if and only if 
$o \in \partial d_p (q)$. Since $\partial d_p(p)$ equals the unit 
closed ball centered at the origin of $T_pM$, 
$o \in \partial d_p (p)$. 
Note that, even if $o \not\in \partial d_p (q)$, 
it is possible to occur that $q$ is a cut point of $p$ in general. 
\end{example}

As a direct consequence 
of Theorem \ref{2015_07_30_theorem2}, we have 

\begin{corollary}\label{2015_07_30_theorem1}
Let $F$ be a bi-Lipschitz homeomorphism 
from a compact Riemannian manifold $M$ 
onto a Riemannian manifold $N$. 
If $F$ and $F^{-1}$ have no singular points 
on $M$ and $N$, respectively, 
then $M$ and $N$ are diffeomorphic.
\end{corollary}

\begin{remark} Shikata would be the first researcher who approximated 
a bi-Lipschitz homeomorphism via diffeomorphisms: 
In \cite{Sh1}, he introduced a distance for a pair 
of compact differentiable manifolds which are bi-Lipschitz homeomorphic 
and proved that if the distance between such a pair is 
smaller than a certain positive constant, 
then the bi-Lipschitz map can be approximated via diffeomorphisms, 
i.e., the manifolds are diffeomorphic. 
Moreover, as an application to the differentiable pinching problem, 
he proved in \cite{Sh2} that there exists a certain constant 
$\delta(n)\in(1/4,1)$ depending on a number $n$ such that 
if sectional curvature of a simply connected, 
compact Riemannian manifold $M$ 
of dimension $n$ is $\delta(n)$-pinched, 
then $M$ is diffeomorphic to the standard sphere. 
What astonishes us is that he defined such a distance 
between two manifolds for getting the differentiable sphere theorem more than ten years before Gromov's Hausdorff distance in \cite{G2}.
\end{remark}

\subsection{Applications of Theorem \ref{2015_07_30_theorem2}: 
Three differentiable sphere theorems}\label{2016_01_11_new}

Now, we are going to state applications 
of Theorem \ref{2015_07_30_theorem2}: 
Taking Theorem \ref{2015_07_27_weinstein} 
into account, we consider the following setting. 
Let $M_i$ ($i=1,2$) be a compact manifold 
of dimension $n$ admitting a point $p_i \in M_i$ 
with a single cut point $q_i \in M_i$, 
and let $d_{M_i}(p_i,q_i)= \ell_i$, 
where $d_{M_i}$ denotes the distance function of $M_i$. 
Note that $M_i$ is homeomorphic to $S^n$. 
Choose a linear isometry $I: T_{p_1}M_1\lra T_{p_2}M_2$, 
where $T_{p_i}M_i$ denotes the tangent space of $M_i$ at $p_i$. 
For each $i=1,2$, let $\sigma^{p_i}_{q_i}$ be the diffeomorphism from 
$\Sph^{n-1}_{p_i}
:=\{v \in T_{p_i}M_i\,|\,\|v\|=1\}$ 
onto $\Sph^{n-1}_{q_i}$ defined by 
\begin{equation}\label{2015_07_30_def_of_subsigma}
\sigma^{p_i}_{q_i} (u_i):= -\dot{\tau}_{u_i} (\ell_i),
\end{equation}
where $\tau_{u_i}(t):= \exp_{p_i}tu_i$ 
for all $u_i \in \Sph^{n-1}_{p_i}$ and all $t \in [0,\ell_i]$. 
Thus, we have the diffeomorphism 
\begin{equation}\label{2015_07_27_def_of_sigma}
\sigma: \Sph^{n-1}_{q_1}\lra \Sph^{n-1}_{q_2}
\end{equation}
defined by 
$\sigma
:= \sigma^{p_2}_{q_2} \circ I \circ \sigma^{q_1}_{p_1}$, 
where $\sigma^{q_1}_{p_1}:= (\sigma_{q_1}^{p_1})^{-1}$. 
The bi-Lipschitz constant $\bL(\sigma)$ of $\sigma$ is given by \begin{equation}\label{2015_08_26_MA_3}
\bL(\sigma):=\inf\{\ell\,|\,\ell^{-1}\|u-v\| \le \|\sigma (u)-\sigma(v)\|
\le \ell\|u-v\| \ {\rm for \ all} \ u,v \in \Sph^{n-1}_{q_1}\}. 
\end{equation}
For a geodesic segment 
$\gamma:[0,\pi]\lra\Sph^{n-1}_{q_1}$ 
with $\|\dot{\gamma}\|:=\|d\gamma/dt\|\equiv 1$, 
let $c:[0,\pi]\lra\Sph^{n-1}_{q_2}$ be the curve defined by 
\begin{equation}\label{2015_09_30_def_of_semi-geodesic}
c:= \sigma\circ \gamma.
\end{equation}
Then, we find that the relation between the differentiable 
structures of $M_1$ and $M_2$ depends 
on the map $\sigma$ above, i.e., 
the first application of Theorem \ref{2015_07_30_theorem2} 
is stated as follows: 

\begin{TA}  
If $\bL(\sigma)$ satisfies 
\begin{equation}\label{2015_07_25_MA_1}
\|\ddot{c}\|^2 - 
2\bL (\sigma)^{-2} 
\le 2\left\{ \frac{\sqrt{2}-1}{2(e^\pi-1)} \right\}^2 -1
\end{equation}
for all geodesic segments 
$\gamma ([0,\pi])\subset \Sph_{q_1}^{n-1}$ 
with $\|\dot{\gamma}\|\equiv 1$ where 
$c:= \sigma\circ \gamma$ and $\|\ddot{c}\|:=\|d^2c/dt^2\|$, 
or if the $\sigma$ and $n$ satisfy  
\begin{equation}\label{2015_08_12_MA_2}
\bL(\sigma)^2 \le 1+ \left\{ \frac{8}{\pi}(n-1)\right\}^{-\frac{1}{2}}, 
\end{equation}
then $M_1$ and $M_2$ are diffeomorphic.
\end{TA}

\begin{remark}\label{2016_01_22_rem1}
We give two remarks on the differentiable 
exotic sphere theorem I and the outlines of the proof 
of it in order to show Theorem \ref{2015_07_30_theorem2} 
to be useful: 

\begin{itemize}
\item The assumption \eqref{2015_07_25_MA_1} does not 
depend on the dimension. Moreover, it 
implies that $\sigma$ is almost an isometry, 
and that $c$ is almost a geodesic segment 
on $\Sph^{n-1}_{q_2}$. 
The inevitability of the restriction of $\ddot c$ is 
discussed in Sect.\,\ref{sect2_new4}. 

\item The right side of \eqref{2015_08_12_MA_2} 
is the same constant as Karcher \cite{K} estimated 
to get a sharper version of Shikata's theorem in \cite{Sh1}.

\item The outline of the proof in the case 
where \eqref{2015_07_25_MA_1} holds: 
Let $d_{M_i}(p_i,q_i):= \pi$ ($i=1,2$) by normalizing the metric. 
For each $(t, u_1) \in [0,\pi] \times \Sph^{n-1}_{p_1}$, 
we define the bi-Lipschitz homeomorphism 
$F: M_1\lra M_2$ by 
\begin{equation}\label{2016_01_20_map}
F(\exp_{p_1}tu_1):=\exp_{p_2}(t I(u_1)).
\end{equation}
Note that $F$ is a diffeomorphism 
between $M_1\setminus \{q_1\}$ 
and $M_2\setminus \{q_2\}$. 
Thanks to \eqref{2015_07_25_MA_1}, 
we see that $q_1$ is a non-singular point of $F$, 
and hence $F$ has no singular points on $M_1$. 
By Theorem \ref{2015_07_30_theorem2}, for a sufficiently small $\eta>0$, we can approximate $F$ via smooth immersions 
$f_\eta: M_1\lra M_2$. By a topological argument, 
we see that the local diffeomorphism $f_\eta$ is a bijection. Therefore, 
$M_1$ and $M_2$ are diffeomorphic. 

\item The outline of the proof in the case 
where \eqref{2015_08_12_MA_2} holds: 
Let $F: M_1\lra M_2$ denote the map defined by \eqref{2016_01_20_map}. And we embed $M_2$ 
into $\R^m$ isometrically, where $m\ge n+1$. 
By \eqref{2015_08_12_MA_2} and 
\cite[Theorem 5.1]{K}, we can find $\delta >0$ such that 
a locally smooth approximation 
$F_\ve^{(q_1)}|_{B_\delta (q_1)}$ of 
$F$ is an immersion into $\R^m$ for all $\ve \in (0,\delta)$, 
where $B_\delta (q_1):=\{x \in M_1\,|\,d_{M_1}(q_1, x)< \delta\}$. 
Then, we have the local diffeomorphism $F_\ve:M_1\lra\R^m$ 
defined by $F_\ve:=(1-\varphi)F +\varphi F_\ve^{(q_1)}$. 
Here, $\varphi$ denotes a smooth function on $M_1$ satisfying 
$0\le\varphi\le1$ on $M_1$, $\varphi\equiv1$ on $\ol{B_r(q_1)}$, 
and $\supp \varphi \subset B_R(q_1)$, where $0<r<R<\delta$. 
Choose a sufficiently small open neighborhood $U$ of 
$M_2$ in $\R^m$ so that the smooth (locally) 
distance projection $\pi:U\lra M_2$ is well-defined. 
Since $F_\ve(M_1)\subset U$ for any sufficiently small $\ve>0$, 
the map $f_\ve:= \pi\circ F_\ve :M_1\lra M_2$ 
is a local diffeomorphism. 
Since $f_\ve$ is bijective, $M_1$ and $M_2$ are diffeomorphic. 
The fundamental course in the proof 
of Theorem \ref{2015_07_30_theorem2} 
has a similar construction of $f_\ve$, where we use the partition 
of unity.  
\end{itemize}
\end{remark}

Moreover, let $\ol{c}:[0,\pi]\lra T_{q_2}M_2$ 
be the smooth curve defined by 
\begin{equation}\label{2016_03_23_geodesic_new}
\ol{c}(t):= c(0)\cos t + \dot{c}(0)\sin t,
\end{equation}
where $c$ is the curve in $\Sph^{n-1}_{q_2}$ 
defined by \eqref{2015_09_30_def_of_semi-geodesic} 
and we set $\dot{c}(t):=dc/dt$. Note that $\|\dot{c}(0)\|\not=1$ 
is possible. Then, we can replace \eqref{2015_07_25_MA_1} with 
the following condition \eqref{2016_03_23_TAII_as}: 

\begin{TAII} If $\sigma$ satisfies 
\begin{equation}\label{2016_03_23_TAII_as}
\angle(\ol{c}(t), c(t)) < \frac{\pi}{2}
\end{equation}
for all geodesic segments 
$\gamma ([0,\pi])\subset \Sph^{n-1}_{q_1}$ 
with $\|\dot{\gamma}\|\equiv 1$, 
then $M_1$ and $M_2$ are diffeomorphic.
\end{TAII}

\begin{remark}
We find \eqref{2016_03_23_TAII_as} in the process for proving 
lemmas that we need for the proof of the differentiable 
exotic sphere theorem I in the case 
where \eqref{2015_07_25_MA_1} holds. 
By \eqref{2016_03_23_TAII_as}, 
$q_1$ is non-singular for the map 
$F$ defined by \eqref{2016_01_20_map}, 
and hence $F$ has no singular points on $M_1$. 
By Theorem \ref{2015_07_30_theorem2}, we get the assertion. 
\end{remark}

Next, we will state the third application of Theorem \ref{2015_07_30_theorem2} and its corollaries: 
Let $M$ be a compact Riemannian manifold of dimension $n$. 
For any two distinct points $p,q\in M$, we set 
\[
D(p):=\{x \in M\,|\, d(p, x)< d(q,x)\}, \quad 
D(q):=\{x \in M\,|\, d(p, x) > d(q,x)\}, 
\]and
\[
E_{p,\,q}:=\{x \in M\,|\, d(p, x)= d(q,x)\},
\]
where $d$ denotes the distance function of $M$. 
Moreover, for a point $x \in M$, we set 
\[
\Crit(x):=\{y \in M\,|\,o \in \partial d_x(y)\}.
\] 
With these notations, we have 

\begin{TB}\label{2015_07_29_theorem}
Take any two distinct points $p,q \in M$. 
If \begin{equation}\label{2015_07_29_theorem_1}
D(p) \cap \Crit(p)= \{p\}, \quad D(q) \cap \Crit(q) =\{q\},
\end{equation}
and if for any geodesic segments $\alpha$ and $\beta$ emanating 
from each point $x \in E_{p,\,q}$ to $p$ and $q$, respectively, 
\begin{equation}\label{2015_07_29_theorem_2}
\angle (\dot{\alpha} (0), \dot{\beta}(0)) > \frac{\pi}{2}
\end{equation}
holds at $x$, then 
\begin{enumerate}[{\rm (T-1)}] 
\item $M$ is a twisted sphere, and 
\item there exists a bi-Lipschitz homeomorphism 
between $M$ and $S^n(1)$ 
which is a diffeomorphism except for the point $q$, where 
$S^n(1):=\{v \in \R^{n+1}\,|\,\|v\|=1\}$. 
\end{enumerate}
Furthermore, we have that 
\begin{enumerate}[{\rm (T-3)}] 
\item 
there exists a diffeomorphism 
$\sigma_q^p :\Sph^{n-1}_p\lra \Sph^{n-1}_q$ 
defined similarly to \eqref{2015_07_30_def_of_subsigma} 
such that if the following condition {\rm (a)} or {\rm (b)} is satisfied, 
then $M$ and $S^n(1)$ are diffeomorphic: 
\begin{enumerate}[{\rm (a)}] 
\item 
$\bL((\sigma_q^p)^{-1})$ satisfies \eqref{2015_07_25_MA_1} 
for all geodesic segments $\gamma ([0,\pi])\subset \Sph^{n-1}_{q}$ 
with $\|\dot{\gamma}\|\equiv 1$, where the curve 
$c:[0,\pi]\lra \Sph^{n-1}_S$ is given by 
$c:= \sigma^{N}_{S} \circ I \circ (\sigma_q^p)^{-1} \circ \gamma$, 
and we put $N:=(0,0,\ldots,1), S:= (0,0,\ldots,-1)\in S^n(1)$, 
or if $(\sigma_q^p)^{-1}$ and $n$ satisfy \eqref{2015_08_12_MA_2}; 
\item 
$(\sigma_q^p)^{-1}$ satisfies \eqref{2016_03_23_TAII_as}. 
\end{enumerate}
\end{enumerate}
Note that the curve 
$\ol{c}:[0,\pi]\lra T_S S^n(1)$ 
in \eqref{2016_03_23_TAII_as} is defined by 
\eqref{2016_03_23_geodesic_new} for the 
$c= \sigma^{N}_{S} \circ I \circ (\sigma_q^p)^{-1} \circ \gamma$.
\end{TB}

\begin{remark}\label{2015_03_31_rem1.6} 
We give two remarks on the differentiable twisted 
sphere theorem and several related topics for it: 
\begin{itemize}
\item Since $M$ is twisted by (T-1), 
$M$ admits a metric such that the cut locus of some point in $M$ is a single point by Weinstein's deformation technique for metrics. 
See Remark \ref{2015_08_08_rem} for details. 

\item The diffeomorphism $\sigma_q^p$ in (T-3) induces a boundary 
diffeomorphism $h_{\sigma_q^p}:S^{n-1}\lra S^{n-1}$ such that 
$M=D^n\cup_{h_{\sigma_q^p}} D^n$, 
where $D^n$ denotes the standard $n$-disc and $S^{n-1}= \partial D^n$ 
(See Sect.\,\ref{sect4}). 
Thus, by taking the contrapositive of (T-3), 
we could define clearly a boundary diffeomorphism 
to get an exotic sphere $\Sigma^n$ ($n>4$) from a twisted one. The fact that there are few explicit examples of such maps 
is worthy of note. E.g., Dur\'an's boundary diffeomorphism 
for an $\Sigma^7$ in \cite{D}.

\item Donaldson and Sullivan \cite{DS} proved that there are smooth $4$-manifolds which are homeomorphic, but not bi-Lipschitz.

\item Let $M$ be a piecewise linear (PL) 
$n$-manifold\footnote{A PL $n$-manifold is, 
by definition, a polyhedron admitting a linear 
triangulation that satisfies that the link of each vertex 
is combinatorially equivalent to the boundary of the 
$n$-simplex.} of dimension $n\ge 5$ 
which has the homotopy type of $S^n$. 
Then, the Stallings-Zeeman theorem 
(\cite{St} together with \cite{Z}) says 
that $M\setminus\{\text{point}\}$ is 
PL-homeomorphic to $\R^n$.

\item For any closed Riemannian manifold $M$ of dimension $n\ge2$, 
Cheeger and Colding \cite{CC} found a positive number 
$\delta (n)$ depending on $n$ such that if $\Ric_M\ge n-1$ and 
$\vol (M) > \vol (S^n(1)) - \delta (n)$, then $M$ is diffeomorphic 
to $S^n (1)$, where $\Ric_M$ denotes Ricci curvature of $M$, 
and $\vol (M)$ denotes the volume of $M$. Note that their 
result is a generalization of the pioneering work 
of Otsu-Shiohama-Yamaguchi \cite{OSY} 
in the sectional curvature case. 
\end{itemize}
\end{remark}

Next, we will state three corollaries of the differentiable twisted sphere theorem: Applying the same argument in the proof of (T-2) to any 
twisted sphere, we have 

\begin{corollary}\label{2016_01_10_cor3} 
For any twisted sphere $\Sigma^n$ of general dimension $n$, 
there exists a bi-Lipschitz homeomorphism between $\Sigma^n$ and $S^n(1)$ which is a diffeomorphism except for a single point. 
In particular, if $n>4$, then there exists such a map between each exotic sphere $\Sigma^n$ and $S^n(1)$.
\end{corollary}

\begin{remark}\label{2016_02_08_rem1} 
Let $D^n$ be the standard $n$-disc and 
$S^{n-1}= \partial D^n$. 
Combining with results of Munkres \cite{Mu}, 
Kervaire-Milnor \cite{KM}, and Cerf \cite{Ce}\footnote{An 
alternative proof of $\Gamma_4=0$ is found in Eliashberg's \cite{E}.} , 
we see that if $n \le6$, then the group 
$\Gamma_n:= {\rm Diff} (S^{n-1})/{\rm Diff} (D^n)$ 
is trivial. Here, ${\rm Diff} (X)$ denotes the 
topological group of orientation preserving 
diffeomorphisms of a smooth manifold $X$. 
Hence, every twisted $n$-sphere of dimension $n \le6$ 
is diffeomorphic to $S^n(1)$. 
\end{remark}

Since $\Gamma_4=0$ as in the above, there is no twisted sphere but the standard sphere. 
Hence, if an exotic $4$-sphere $\Sigma^4$ exists, then $\Sigma^4$ is not twisted. 
Thus, we have 

\begin{corollary}\label{2016_01_10_cor1}
If $\Sigma^4$ exists, then $\Sigma^4$ does not satisfy \eqref{2015_07_29_theorem_1}, or \eqref{2015_07_29_theorem_2}. 
\end{corollary}

\begin{remark} So far many topological 
$4$-spheres regarded as exotic are indeed standard. E.g., Akbulut's \cite{A} and 
Gompf's \cite{G}. We refer to \cite{FGMW} about the difficulty to solve SPC$4$. We can find an interesting 
work on the relationship between 
Stein $4$-manifolds and SPC$4$ 
in Yasui's \cite{Y}. 
\end{remark}

Let $M$ be a Grove-Shiohama type $n$-sphere, 
and let $p, q\in M$ such that $d(p, q)= \diam (M)$, 
where $d$ denotes the distance function of $M$. 
Then, Toponogov's comparison theorem shows 
the following process: First $M \setminus B_{\pi/2} (p)$ 
is convex, where 
$B_{\pi/2} (p):=\{x \in M\,|\,d(p,x) < \pi/2\}$, 
second $q$ is a unique point, and finally $\Crit(p) =\{p, q\}$ (See \cite{GS}, or \cite{Kon}, \cite{KO} for details). 
By the same process above, we have $\Crit(q) =\{p, q\}$ since $K_M\ge1$ everywhere. Thus, $M$ satisfies \eqref{2015_07_29_theorem_1} and \eqref{2015_07_29_theorem_2}.
By the differentiable twisted sphere theorem, 
we have 

\begin{corollary}\label{2016_01_10_cor2}
Let $M$ be a Grove-Shiohama type $n$-sphere, 
and let $p, q\in M$ such that $d(p, q)= \diam (M)$. 
Then, there exists a bi-Lipschitz homeomorphism between $M$ and $S^n(1)$ which is a diffeomorphism except for $q$.
\end{corollary}

\begin{remark}\label{2016_01_10_rem1}
We give a remark on Corollary \ref{2016_01_10_cor2} 
and two related results for it: 
\begin{itemize}
\item Since any Grove-Shiohama type $n$-sphere $M$ 
is twisted, $M$ of dimension $n \le6$ is diffeomorphic to $S^n(1)$. 
See  Remark \ref{2016_02_08_rem1} above. 

\item Grove and Wilhelm \cite{GW} proved that 
every Grove-Shiohama type $n$-sphere 
is diffeomorphic to $S^n (1)$ if the diffeomorphism 
stability question is ``yes" in their sense.

\item Petersen and Wilhelm \cite{PW} announced that 
the Gromoll-Meyer exotic sphere in \cite{GM} admits a metric with positive sectional curvature everywhere.
\end{itemize}
\end{remark}


\medskip

The organization of this article is as follows: 
Sect.\,\ref{sect2} has three subsections. 
In Sect.\,\ref{2016_03_12_sect2_sub1}, 
we study an approximation of a Lipschitz function 
on a Riemannian manifold $M$. 
Using non-smooth analysis, 
we prove, as a main lemma (Lemma \ref{lemN12}), 
that \eqref{2015_07_29_theorem_2} also holds 
for two gradient vector fields on a compact set in $M$ 
of smooth approximations of two distance functions 
from distinct points in $M$. In Sect.\,\ref{2015_08_26_Sect3.2}, we establish 
an approximation method for a Lipschitz map from a compact Riemannian
manifold into a Riemannian manifold using some techniques 
from non-smooth analysis and the partition of unity, 
and prove our main theorem, 
Theorem \ref{2015_07_30_theorem2}, 
applying the method. In Sect.\,\ref{sect3}, we prove Corollary \ref{2015_07_30_theorem1} 
applying Theorem \ref{2015_07_30_theorem2}. 
Sect.\,\ref{sect2_new} has four subsections. 
In the first two subsections \ref{sect2_new1} and \ref{sect2_new2}, 
we give preliminaries to proofs 
of the differentiable exotic sphere theorems I and II 
in assuming each of the conditions, i.e., 
\eqref{2015_07_25_MA_1}, \eqref{2015_08_12_MA_2}, and \eqref{2016_03_23_TAII_as}. In Sect.\,\ref{sect2_new3}, 
we prove the two theorems. 
In Sect.\,\ref{sect2_new4}, 
we discuss why we need the restriction of 
$\ddot{c}$ in \eqref{2015_07_25_MA_1} for the theorem I. 
In Sect.\,\ref{sect4}, we prove the differentiable twisted sphere theorem applying Lemma \ref{lemN12} and 
Theorem \ref{2015_07_30_theorem2}. 

\section{Approximations of Lipschitz maps via immersions}\label{sect2}

We prove here our main theorem, Theorem \ref{2015_07_30_theorem2}.  For the proof, we establish an approximation method for a Lipschitz map from a compact Riemannian manifold into a Riemannian manifold 
using non-smooth analysis. 
As was mentioned in Remark \ref{2016_03_11_rem}, 
this approximation method generalizes the whole method 
of Grove-Shiohama \cite{GS}. 
Note that we do not assume curvature assumptions at all, 
and that the smoothing technique bases on the partition of unity, 
while those in \cite{GS}, \cite{K} depend on the center of mass technique constructed by Grove and Karcher \cite{GK}. 

\subsection{Approximations of Lipschitz functions}\label{2016_03_12_sect2_sub1}

We first treat an approximation of a Lipschitz function $f$ 
on a Riemannian manifold $M$ of dimension $n$. 
As a main lemma (Lemma \ref{lemN12}) in this subsection, 
we prove that our approximation method keeps \eqref{2015_07_29_theorem_2} 
for two gradient vector fields 
$\nabla (d_p)_\ve, \nabla (d_q)_\ve$ on a compact set in $M$ 
of smooth approximations $(d_p)_\ve, (d_q)_\ve$ 
of two distance functions $d_p, d_q$ 
from distinct points $p,q\in M$, 
where $d$ denotes the distance function of $M$. 
This lemma will be applied 
to the proof of the differentiable twisted sphere theorem 
in Sect.\,\ref{sect4}.\par 
Take any point $p \in M$, 
and fix it. Since the exponential map $\exp_p$ on $T_pM$ is a diffeomorphism from $\B_r (o_p)$ onto $B_r(p)$ for a sufficiently small $r > 0$, we denote by $\exp_p^{-1}$ the inverse map of $\exp_p|_{\B_r (o_p)}$. Here we set 
\[
\B_r (o_p):=\{v \in T_pM\,|\, \|v\| < r\}, \quad 
B_r(p) :=\{q \in M \,|\, d(p, q) < r\},
\] 
where $o_p$ denotes the origin of $T_pM$. 
In what follows, we identify $T_pM$ 
with Euclidean $n$-dimensional space 
$(\R^n, \langle\,\cdot\,,\,\cdot\,\rangle)$. 
Then, we may define a smooth approximation 
of the $f$ around $p$: 

\begin{definition}\label{2014_05_23_def5.1}
For each $\ve >0$, let $f^{(p)}_\ve : B_r (p) \lra \R$ denote the function defined by  
\[
f^{(p)}_\ve (q) :=\int_{\R^n} f(q(y)) \rho_\ve (y) dy
= \int_{\R^n} f(\exp_p(y)) \rho_\ve (\exp_p^{-1}q-y) dy,
\]
where we set $q(y):= \exp_p (\exp_p^{-1}q -y)$, 
and the function $\rho_\ve$ denotes the mollifier. 
Refer to \cite{Sh1}, \cite{K}, or \cite{GS} for details of the mollifier.
\end{definition}

Choose a locally finite covering $\{B_{r_i}(p_i)\}$ 
of strongly convex balls of $M$ such that 
\[
M= \cup_i B_{r_i/2}(p_i).
\] 
Here, a subset $C \subset M$ is said to be {\em strongly convex} 
if for any two points $x,y$ in the closure $\ol{C}$ of $C$, 
there exists a unique minimal geodesic segment 
$\gamma :[0,1]\lra M$ emanating from $x= \gamma (0)$ 
to $y =\gamma (1)$ such that $\gamma ((0,1))\subset C$. Note that 
for each $x \in M$, there exists a strongly convex ball of its center $x$ 
(Whitehead's convexity theorem \cite{Wh}).\par 
Let $\{\psi_i\}_{i=1}^\infty$ be the partition 
of unity subordinate to $\{B_{r_i/2}(p_i)\}$. 
Define the smooth approximation $f_\ve$ of the $f$ on $M$ by
\[
f_\ve(q):=\sum_{i=1}^\infty\psi_i(q)f_\ve^{(p_i)}(q).
\]
By definition, 
\begin{equation}\label{eq:N1.0}
f^{(p_i)}_\ve(q)=\int_{\R^n}f(q_i(y))\rho_\ve(y)dy
\end{equation}
for $q\in B_{r_i}(p_i)$,
where
\begin{equation}\label{2014_07_29_eq:No2.2}
q_i(y):=\exp_{p_i}(\exp_{p_i}^{-1}q-y).
\end{equation}

\begin{lemma}\label{lemN5}
For each center $p_i$ of $B_{r_i} (p_i)$, 
\begin{equation}\label{2014_08_18_lemN5_1}
\sup_{q\in\supp\psi_i}|f^{(p_i)}_\ve(q)-f(q)| 
\le \ve \cdot \L (f) \cdot \L(\exp_{p_i}|_{ \B_{r_i}(o_{p_i})})
\end{equation}
holds for all $\ve\in(0,\ve_i)$, 
where $\ve_i:=r_i-\max\{ \| \exp_{p_i}^{-1}q \| \; | \: q\in\supp\psi_i\}$,
and 
$\L(f)$ and $\L(\exp_{p_i} |_{\B_{r _i} (o_{p_i})})$ denote the Lipschitz constants of 
$f$ and $\exp_p|_{\B_{r_i } (o_{p_i})}$, respectively, i.e.,  
\begin{equation}\label{2014_05_25_lem5.2_1}
\L(f)
:= \sup \left\{ 
\frac{|f(q_1) -f(q_2)|}{d(q_1, q_2)}\,\bigg| \,q_1, q_2 \in M, \ q_1\not=q_2
\right\}
\end{equation}
and 
\[
\L(\exp_{p_i} |_{\B_{r _i} (o_{p_i})}):= \sup \left\{\frac{d(\exp_{p_i} v, \exp_{p_i} w)}{\|v-w\|}\,\bigg|\,
v, w \in \B_{r_i} (o_{p_i}), \ v\not=w \right\}.
\]
\end{lemma}

\begin{proof}
Take any point $q\in\supp\psi_i$. Then, for any $y\in\B_{\ve_i}(o_{p_i})$, 
we have, by the triangle inequality,
\[
\|\exp_{p_i}^{-1}q-y \| \le \| \exp_{p_i}^{-1}q \|+\|y\|<r_i.
\] 
Thus, $\exp_{p_i}^{-1}q-y$ and $\exp_{p_i}^{-1}q$ are elements of $\B_{r_i}(o_{p_i})$ for all $y\in\B_{\ve_i}(o_{p_i})$. Hence,
\begin{equation}\label{eq:N2.3}
d(q_i(y),q) 
\le 
\L (\exp_{p_i}|_{\B_{r_i}(o_{p_i})}) \| (\exp_{p_i}^{-1}q-y) -\exp_{p_i}^{-1}q \| 
= \L (\exp_{p_i}|_{\B_{r_i}(o_{p_i})}) \|y\|.
\end{equation}
On the other hand, since $\int_{\R^n} \rho_\ve (y)dy = 1$, we have  
$f(q) = \int_{\|y\| < \ve} f(q) \rho_\ve (y)dy$. Then, 
\begin{align}
|f^{(p_i)}_\ve(q) -f(q)|
&\le \int_{\|y\| <\ve} |f(q_i(y)) -f(q)|\,\rho_\ve (y) \,dy\label{2014_05_25_lem5.2_proof1}\\
& \le \L(f) \int_{\|y\| <\ve} d(q_i(y), q)\,\rho_\ve (y) \,dy,\notag 
\end{align}
where note that $\| y\| < \ve$ when $\rho_\ve(y) \not= 0$.
Combining \eqref{eq:N2.3} with \eqref{2014_05_25_lem5.2_proof1}, we get \eqref{2014_08_18_lemN5_1}.$\qedd$
\end{proof}

Take any $p \in M$ and any $p_i$ with $p\in \supp \psi_i$, and fix them in the following. Moreover, we assume $q\in\supp \psi_i$ in this situation. 

\begin{lemma}\label{lemN6}
For any $\tilde u \in \Sph^{n-1}_q:=\{v \in T_qM\,|\, \|v\|=1 \}$, 
\begin{equation}\label{eq:N1}
(df_\ve^{(p_i)} )_q(\tilde u)=\int_{\R^n}df_{q_i(y)}
(Y_y^{(\tilde u)}(1) ) \rho_\ve (y) dy
\end{equation}
holds,
where
$q_i(y) \in M$ is by definition \eqref{2014_07_29_eq:No2.2}, and
\[
Y_y^{(\tilde u)}(t)
:=\frac{\partial}{\partial s}\exp_{p_i}t (\exp_{p_i}^{-1}(\exp_q s\tilde u)-y)\bigg|_{s=0}
\] 
is a Jacobi field along the geodesic $\exp_{p_i}t(\exp_{p_i}^{-1}q-y)$ 
for each $y$. 
\end{lemma}

\begin{proof}
Since 
\[
(df^{(p_i)}_\ve)_q(\tilde u)=\frac{d}{d s} 
f_\ve^{(p_i)}(\exp_q s\tilde u)\bigg|_{s=0},
\] 
we get \eqref{eq:N1} by \eqref{eq:N1.0}.$\qedd$
\end{proof}

For a pair of points $q_1$ and $q_2$ of $M$ admitting a unique minimal geodesic segment $\gamma$, let $\tau_{q_2}^{q_1}: T_{q_1}M \lra T_{q_2}M$ denote the parallel transportation along $\gamma$. In what follows, 
we will omit brackets of $\tau_{q_2}^{q_1}(u)$ ($u \in T_{q_1}M$) 
for simplicity, i.e, $\tau_{q_2}^{q_1}u:= \tau_{q_2}^{q_1}(u)$. 

\begin{lemma}\label{lemN7} Let $\wt{U}_{q_i(y)}:=\tau_{q_i(y)}^q \tilde{u}$. 
Then, we have
\[
\left| 
(df_\ve^{(p_i)})_q(\tilde u)-\int_{\R^n}df_{q_i(y)}(\wt{U}_{q_i(y)})\rho_\ve(y)dy
\right| 
\le 
\L (f) 
\sup_{y\in \B_\ve (o_{p_i})} 
\|Y^{(\tilde u)}_y(1)-\wt{U}_{q_i(y)}\|.
\]
\end{lemma}

\begin{proof}
By Lemma \ref{lemN6} and \eqref{2014_05_25_lem5.2_1},
we get
\begin{align*}
\left| 
(df_\ve^{(p_i)})_q(\tilde u)-\int_{\R^n}df_{q_i(y)}(\wt{U}_{q_i(y)})\rho_\ve(y)dy
\right| 
&\le \int_{\R^n}|df_{q_i(y)}(Y_y^{(\tilde u)}(1)-\wt{U}_{q_i(y)})|\rho_\ve(y)dy\\[1mm]
&\le \L(f)\sup_{y\in \B_\ve (o_{p_i})} \| Y^{(\tilde u)}_y(1)-\wt{U}_{q_i(y)} \|.
\end{align*}
$\qedd$
\end{proof}

\begin{lemma}\label{lemN8}
For any $\eta>0$, there exists a number $\ve_i(\eta) > 0$ such that 
\[
\sup\{\| Y^{(\tilde u)}_y(1)-\wt{U}_{q_i(y)} \|\; |\:  y\in \B_{\ve_i(\eta)}(o_{p_i}),\,q \in \supp \psi_i,\, \tilde u\in \Sph_q^{n-1}\}<\eta.
\]
\end{lemma}

\begin{proof}
Let $\eta>0$ be an arbitrary number.
Take any $q\in\supp\psi_i$ and any $\tilde u\in\Sph_q^{n-1}$. 
Since $Y^{(\tilde u)}_y(1)=\wt{U}_{q_i(y)}=\tilde u$ for $y=o$, 
there exists a positive number $\ve(q,\tilde u,p_i,\eta) >0$ such that
$
\|Y^{(\tilde u)}_y(1)-\wt{U}_{q_i(y)}\|<\eta
$
for all $y\in \B_{\ve(q,\,\tilde u,\,p_i,\,\eta)}(o_{p_i})$. 
Since
$\supp\psi_i$ and $\Sph_q^{n-1}$
are compact, there exists a positive number $\ve_i(\eta)$ such that 
$\|Y^{(\tilde u)}_y(1)-\wt{U}_{q_i(y)}\|<\eta$ 
for all $q\in\supp \psi_i$, $\tilde{u} \in\Sph_q^{n-1}$, 
and $y\in \B_{\ve_i(\eta)}(o_{p_i})$.
$\qedd$
\end{proof}

For each $v\in \R^n=T_pM$, let $\pi_p(v)$ be the nearest point on 
$\partial f(p)$ from $v$. 
Since $\partial f(p)$ is convex, the nearest point is uniquely determined.
Moreover, it is easy to check that the map $\pi_p : \R^n\lra \partial f(p)$ is continuous.

\begin{lemma}\label{lemN9} Set 
\begin{equation}\label{2014_08_22_lemN9_1}
v^{(i)}_\ve:=\int_{\R^n}\pi_p(\tau_p^q\tau_q^{q_i(y)}\nabla f (q_i(y)) ) \rho_\ve(y)dy
\end{equation}
and $u:=\tau_p^q\tilde u$. Then, we have  
\[
\left| 
\int_{\R^n}df_{q_i(y)}(\wt{U}_{q_i(y)})\rho_\ve (y)dy-\langle v^{(i)}_\ve,u\rangle
\right|
\le
\int_{\R^n} \|(1-\pi_p)(\tau_p^q\tau_q^{q_i(y)}\nabla f (q_i(y)) ) \|\rho_\ve(y)dy,
\]
where we set $\wt{U}_{q_i(y)} := \tau^q_{q_i(y)}\tilde{u}$. 
\end{lemma}

\begin{proof} It is easy to check that 
$
df_{q_i(y)}(\wt{U}_{q_i(y)})
=\langle \tau^q_p\tau_q^{q_i(y)}\nabla f (q_i(y)) ,u\rangle
$. Thus, by the Cauchy-Schwarz inequality, we get
\begin{align*}
\left| 
\int_{\R^n}df_{q_i(y)}(\wt{U}_{q_i(y)})\rho_\ve(y)dy-\langle v^{(i)}_\ve,u\rangle
\right|
&=
\left| \int_{\R^n}\langle(1-\pi_p)( \tau_p^q\tau_q^{q_i(y)}\nabla f (q_i(y)) ),u\rangle \rho_\ve(y)dy \right|\\[1mm]
&\le \int_{\R^n} \| (1-\pi_p)(\tau_p^q\tau_q^{q_i(y)}\nabla f (q_i(y)) ) \|\rho_\ve(y)dy.
\end{align*}
$\qedd$
\end{proof}

\begin{definition}
A map $F:\R^n\lra\R^k$ is called a {\it locally $L^1$-map} 
if each $F^i:\R^n\lra \R$ is a locally $L^1$-function, 
where $F^i(x)$ denotes the $i$-th component of $F(x)\in \R^k$. 
\end{definition}

\begin{lemma}\label{lemN4}
Let $F:\R^n\lra \R^k$ be a locally $L^1$-map, 
and $C \subset \R^k$ a compact convex set. If $F(\R^n)\subset C$, 
then $\int_{\R^n}\rho(x)F(x)dx \in C$ for all non-negative continuous function $\rho: \R^n \lra \R$ whose support is compact with $\int_{\R^n}\rho(x)dx=1$.
\end{lemma}

\begin{proof}
Let $\{H_\alpha\}_{\alpha\in\Gamma}$ be the family 
of all closed half spaces in $\R^k$. It is well known that 
$K=\bigcap_{K\subset H_\alpha}H_\alpha$ 
for any closed convex subset $K$ of $\R^k$. 
Choose any closed half space $H_\alpha$ 
with $C\subset H_\alpha$. Let $n_\alpha$ 
denote the inward pointing unit normal vector 
of $H_\alpha$ and $v_\alpha\in \R^k $ a nearest point of $C$ 
from $\partial H_\alpha$. 
Then, $\langle F(x)-v_\alpha,n_\alpha\rangle\geq0$ holds for all $x$, since $F(x)\in C$.
Since $\int_{\R^n}\rho(x)dx=1$, $\int_{\R^n}\rho(x)v_\alpha dx=v_\alpha$. 
Thus,
\[
\left\langle 
\int\rho(x)F(x)dx-v_\alpha,n_\alpha
\right\rangle
=\int\rho(x)\langle F(x)-v_\alpha,n_\alpha\rangle dx\geq0,
\]
and hence $\int_{\R^n}\rho(x)F(x)dx\in H_\alpha$ for all $H_\alpha$ with $C\subset H_\alpha$. 
Thus, $\int_{\R^n}\rho(x)F(x)dx\in C$.$\qedd$
\end{proof}

\begin{lemma}\label{lemN1}
For any $\eta>0$, there exists a number $\delta(p,\eta)>0$ such that 
\[
\tau_p^{q_1}\tau_{q_1}^{q_2}\nabla f(q_2) \in \partial f(p)_\eta
:=\bigcup_{v\in\partial f(p)}\B_{\eta}(v)
\] 
holds for all $q_1\in B_{\delta(p,\eta)}(p)$ 
and $q_2\in B_{\delta(p,\,\eta)}(p)\cap(M\setminus E_f)$.
\end{lemma}

\begin{proof}
Take any two sequences $\{x_i\}$ of $x_i \in M$ and 
$\{y_i\}$ of $y_i \in M\setminus E_f$ both of which are convergent to the point $p$. 
If the limit $\lim_{i\to\infty}\nabla f (y_i)$ exists, then 
$\lim_{i\to\infty}\tau_p^{x_i}\tau_{x_i}^{y_i} \nabla f (y_i) 
=\lim_{i\to\infty} \nabla f (y_i)\in\partial f(p)$. 
This implies the existence of the positive number $\delta(p,\eta)$.$\qedd$
\end{proof}

\begin{lemma}\label{lemN10}
For any $\eta>0$, there exist numbers 
$\delta_1(p,\eta)>0$ and $\ve(p,p_i,\eta)>0$ such that
\[
\tau_p^q \nabla f_\ve^{(p_i) }  (q)\in \partial f(p)_\eta
\]
for all $q\in B_{\delta_1(p,\,\eta)}(p)\cap\supp\psi_i$, 
and for all $\ve\in(0,\ve(p,p_i,\eta))$.
\end{lemma}

\begin{proof}
It follows from Lemmas \ref{lemN7} and \ref{lemN8} that
\begin{equation}\label{eq:N4}
\left|
(df_\ve^{(p_i)})_q(\tilde u)-\int_{\R^n}df_{q_i(y)}(\wt{U}_{q_i(y)})\rho_\ve(y)dy
\right|
<\frac{\eta}{2}
\end{equation}
for all $\ve\in (0,\ve_i(\eta/2{\rm Lip}(f)))$, 
$q\in\supp\psi_i$, and $\tilde u\in\Sph_q^{n-1}$.
If $\ve$ is not greater than
\[
\min \left\{ 
r_i-\max_{q\in\supp\psi_i } \| \exp_{p_i}^{-1}q \|, \ \  \frac{\delta(p,\eta/2)}{2\L(\exp_{p_i}|_{\B_{r_i}(o_{p_i})})}
\right\},
\]
then
$\| \exp_{p_i}^{-1}q-y \| \le \| \exp_{p_i}^{-1}q \| + \|y\| < r_i$ 
for all $y\in\B_\ve(o_{p_i})$. 
Hence, 
\[
d(q_i(y),q) \le \L( \exp_{p_i}|_{\B_{r_i}(o_{p_i})} ) \|y \| 
< \frac{\delta(p,\eta/2)}{2}
\] 
for all $y\in\B_\ve(o_{p_i})$. 
Thus, we have 
\[
d(p,q_i(y))\leq d(p,q)+d(q,q_i(y))
<
\frac{\delta(p,\eta/2)}{2} +\frac{\delta(p,\eta/2)}{2}=\delta(p,\eta/2)
\]
for all $q\in\supp\psi_i\cap B_{\delta_1(p,\eta)}(p)$, 
where we set $\delta_1(p,\eta):=\delta(p,\eta/2)/2$. Let
\[
\ve(p,p_i,\eta)
:=
\min 
\left\{
\ve_i( \eta/2 \L(f) ), \ \ 
\frac{\delta_1(p,\eta)}{\L (\exp_{p_i}|_{\B_{r_i}(o_{p_i})})}, \ \ 
r_i-\max_{q\in\supp\psi_i}\|\exp_{p_i}^{-1}q\|
\right\}.
\]
From Lemma \ref{lemN1}, we get
\[
\| 
(1-\pi_p)(\tau_p^q\tau_q^{q_i(y)}\nabla f (q_i(y)))
\| 
<\frac{\eta}{2}
\] 
for all $q\in B_{\delta_1(p,\,\eta)}(p)$ and almost all $y\in\B_{\ve(p,\,p_i,\,\eta)}(o_{p_i})$. 
Therefore, by Lemma \ref{lemN9}, 
we obtain
\begin{equation}\label{eq:N7}
\left| 
\int_{\R^n}df_{q_i(y)}(\wt{U}_{q_i(y)})\rho_\ve(y)dy-\langle v^{(i)}_\ve,u\rangle
\right| 
<\frac{\eta}{2}
\end{equation}
for all $q\in B_{\delta_1(p,\,\eta)}(p)$ and all $\ve\in(0,\ve(p,p_i,\eta))$. 
By the triangle inequality and the equations \eqref{eq:N4}, \eqref{eq:N7}, 
we obtain
\begin{equation}\label{eq:N8}
|(df_\ve^{(p_i)})_q(\tilde u)-\langle v^{(i)}_\ve, u\rangle|<\eta
\end{equation}
for all $\tilde u\in\Sph^{n-1}_q$, $q\in\supp\psi_i\cap B_{\delta_1(p,\,\eta)}(p)$, 
and $\ve\in(0,\ve(p,p_i,\eta))$. Since 
$(df_{\ve}^{(p_i)})_q(\tilde u)=\langle\tau_p^q \nabla {f^{(p_i)}_\ve}(q), u\rangle$ and $\tilde{u}$ is arbitrarily chosen, we have, by \eqref{eq:N8},
\[
\|
\tau_p^q {\nabla {f_\ve}^{(p_i)}} (q)-v^{(i)}_\ve
\|
= 
\frac{
\left|\langle \tau_p^q \nabla {f^{(p_i)}_\ve}(q) -v^{(i)}_\ve, 
\tau_p^q \nabla {f^{(p_i)}_\ve}(q) -v^{(i)}_\ve\rangle\right|
}
{
\|
\tau_p^q {\nabla {f_\ve}^{(p_i)}} (q)-v^{(i)}_\ve
\|
}
<\eta
\] 
for all $q\in \supp\psi_i\cap B_{\delta_1(p,\,\eta)}(p)$ and $\ve\in(0,\ve(p,p_i,\eta))$. 
By Lemma \ref{lemN4} and \eqref{2014_08_22_lemN9_1},
$v_\ve^{(i)}\in \partial f(p)$. 
Hence, $\tau_p^q \nabla {f_\ve}^{(p_i)} (q) \in \partial f(p)_\eta$
for all $q\in \supp\psi_i\cap B_{\delta_1(p,\,\eta)}(p)$, and $\ve\in(0,\ve(p,p_i,\eta))$.
$\qedd$
\end{proof}

\begin{lemma}\label{lemN11}
For any $\eta>0$, 
there exist numbers $\delta_2(p,\eta)>0$ and $\ve(p,\eta)>0$ such that 
\[
\tau_p^q \nabla f_\ve (q) \in  \partial f(p)_\eta
\]
for all $q\in B_{\delta_2(p,\,\eta)}(p)$ and all $\ve\in(0,\ve(p,\eta))$.
\end{lemma}

\begin{proof}
By Lemma \ref{lemN10}, $\tau_p^q \nabla f_\ve^{(p_i)}(q) \in \partial f(p)_{\eta/2}$ 
holds for all $q\in B_{\delta_1(p,\,\eta/2)}(p)\cap\supp\psi_i$ and $\ve \in(0,\ve(p,p_i,\eta/2))$. 
Since $\partial f(p)_{\eta/2}$ is convex,
\begin{equation}\label{eq:N10}
\sum_i\psi_i(q) \cdot \tau_p^q \nabla f_\ve^{(p_i)} (q)\in \partial f(p)_{\eta/2}
\end{equation}
for all $q\in B_{\delta_1(p,\,\eta/2)}(p)$ and $\ve\in(0,\ve_1(p,\eta))$, 
where
\[
\ve_1(p,\eta):=\min\{\ve (p,p_i,\eta/2)\; | \: B_{\delta_1(p,\,\eta/2)}(p)\cap\supp\psi_i\ne\emptyset\}.
\]
Fix any $q\in B_{\delta_2(p,\,\eta)}(p)$, 
where $\delta_2(p,\eta):=\delta_1(p,\eta/2)$, 
and take any $\tilde u \in \Sph^{n-1}_{q}$. Since $\sum_i\psi_i=1$, $\sum_id\psi_i(\tilde u)=0$ holds. 
Then, we get
\begin{equation}\label{eq:N11}
(df_\ve)_q(\tilde u)=\sum_i\psi_i(q)df_\ve^{(p_i)}(\tilde u)+\sum_id\psi_i(\tilde u)(f_\ve^{(p_i)}(q)-f(q))
\end{equation}
By Lemma \ref{lemN5}, 
\[
\left| 
\sum_id\psi_i(\tilde u) (f^{(p_i)}_\ve (q)-f(q)) 
\right| 
\le \ve \sum_i 
|d\psi_i(\tilde u)| \cdot \L(f) \cdot 
\L( \exp_{p_i} |_{\B_{r_i}(o_{p_i})})
\]
holds. Therefore,
there exists a number $\ve(p,\eta)\in(0,\ve_1(p,\eta)]$ such that
\begin{equation}\label{eq:N12}
\left| 
\sum_id\psi_i(\tilde u)(f^{(p_i)}_\ve(q)-f(q))
\right|
<\frac{\eta}{2}
\end{equation}
for all $\ve\in(0,\ve(p,\eta))$ and for all $q\in B_{\delta_2(p,\,\eta)}(p)$. 
By \eqref{eq:N11} and \eqref{eq:N12},
\[
\left| 
\langle \nabla f_\ve (q),\tilde u\rangle
- 
\left\langle\sum_i\psi_i(q) \nabla f_\ve^{(p_i)} (q),\tilde u \right\rangle
\right| 
<\frac{\eta}{2}.
\]
If we set $u:=\tau_p^q\tilde u$, then
\[
\left|\langle\tau_p^q \nabla f_\ve (q),u\rangle
-
\left\langle\sum_i\psi_i(q) \cdot \tau_p^q \nabla f_\ve^{(p_i)}
(q),u\right\rangle\right|
<\frac{\eta}{2}.
\] 
Since $u$ is any unit tangent vector,
\[
\left\| 
\tau_p^q \nabla f_\ve(q)
-\sum_i\psi_i(q) \cdot \tau_p^q \nabla f_\ve^{(p_i)} (q)
\right\|
<\frac{\eta}{2}.
\] 
Hence, from \eqref{eq:N10}, it follows that $\tau_p^q \nabla f_\ve(q)\in\partial f(p)_\eta$ for all $q\in B_{\delta_2(p,\,\eta)}(p)$ and all $\ve\in(0,\ve(p,\eta))$.$\qedd$
\end{proof}

\begin{lemma}\label{lemN2}
Let $A, B \subset \R^n\setminus \{o\}$ be compact sets such that $\angle(u,v)>\pi/2$ holds 
for all $u\in A$ and $v\in B$. 
Then, $o \notin \Conv(A)$ and $o\notin\Conv (B)$.  
Moreover, $\angle(u,v)>\pi/2$ also holds for all $u\in \Conv(A)$ and $v\in \Conv(B)$. 
\end{lemma}

\begin{proof}
Take any $u\in A$, and fix it. Since $\angle (u, v) > \pi/2$ for all $v \in B$, $B$ is a subset 
of the open convex cone $V_u:=\{ v \in \R^n\setminus \{o\}\,|\, \angle (u,v) > \pi/2\}$. 
Thus, $B \subset \bigcap_{u \in A} V_u$. 
The set $\bigcap_{u\in A}V_u$ is convex, since each $V_u$ is convex. 
From the definition of the convex hull, $\Conv (B) \subset \bigcap_{u \in A} V_u$. 
Hence, we have proved that $o \not\in \Conv (B)$, and that 
$\angle(u,v) >\pi/2$ for all $u \in A$ and all $v \in \Conv(B)$. 
On the other hand, choose any $v \in \Conv (B)$, and fix it. 
Since $\angle(u,v) >\pi/2$ for all $u \in A$, $A$ is a subset of the open convex cone $V_v$. 
Then, $A \subset \bigcap_{v \in \Conv (B)}V_v$, 
and hence $\Conv (A) \subset \bigcap_{v \in \Conv (B)}V_v$. In particular, 
$o \not\in \Conv(A)$ and $\angle(u,v) >\pi/2$ 
for all $u \in \Conv(A)$ and all $v \in \Conv (B)$.$\qedd$ 
\end{proof}

\begin{lemma}\label{lemN3}
Let $f$ and $h$ be Lipschitz functions on $M$. 
Assume that $o\notin\partial f(x)$, $o\notin\partial h(x)$ 
on a compact set $K \subset M$. If $\angle (u,v)>\pi/2$ holds 
for all $x\in K$ and all $u\in\partial f(x)$, $v\in\partial h(x)$, 
then there exists a number $\eta_K>0$ such that, 
for any $x\in K$ and any $u\in \partial f(x)_{\eta_K}$, $v\in \partial h(x)_{\eta_K}$, 
$\angle (u,v)>\pi/2$ holds. 
\end{lemma}

\begin{proof}
By supposing that the conclusion is false, we will get a contradiction. 
Then, for each positive integer $i$, there exist a point $x_i \in K$, 
$u_i \in \partial f (x_i)$, and $v_i \in \partial h (x_i)$ 
such that $\angle (u_i, v_i) < \pi/2 + 1/i$. 
Since $K$ is compact, we can assume, 
taking a subsequence, if necessary, 
that $x:= \lim_{i\to\infty} x_i \in K$ exists. By definition, 
any limits of the sequences $\{\tau_x^{x_i} u_i\}$ and 
$\{\tau_x^{x_i} v_i\}$ are elements of $\partial f(x)$ and $\partial h(x)$, respectively. 
Since $\angle (\tau_x^{x_i} u_i, \tau_x^{x_i} v_i) = \angle (u_i, v_i) < \pi/2 + 1/i$ 
for each $i$, we get $\angle (u, v) \le \pi/2$ for some $u \in \partial f(x)$ 
and $v \in \partial h(x)$. This is a contradiction.$\qedd$
\end{proof}

\begin{lemma}\label{lemN12}
Let $K\subset M$ be a compact set. 
Assume that there exist two distinct points $p, q \in M\setminus K$ 
such that for any minimal geodesic segments $\alpha$ and $\beta$ 
emanating from each point $x\in K$ to $p$ and $q$, respectively, 
$\angle (\dot{\alpha}(0), \dot{\beta}(0))>\pi/2$ holds. 
Then, for any sufficiently small $\ve>0$, 
$\angle (\nabla (d_p)_\ve, \nabla (d_q)_\ve) > \pi/2$ holds on $K$, where $d_p(x):= d(p, x)$ for all $x \in M$.
\end{lemma}

\begin{proof}
Let $x$ be any point of $K$. 
It follows from Lemma \ref{lemN2}
that $o\notin\partial d_p(x)\cup\partial d_q(x)$ and $\angle(u,v) > \pi/2$ holds for all $u\in\partial d_p(x)$ and $v\in\partial d_q(x)$. 
Then, by Lemma \ref{lemN3},
there exists a positive number $\eta_K$ such that, 
for any $x\in K$, $u\in \partial d_p(x)_{\eta_K}$, 
and $v\in\partial d_q(x)_{\eta_K}$, $\angle(u,v) > \pi/2$ holds again. 
It follows from Lemma \ref{lemN11} that, for each point $x\in K$, 
there exist two numbers $\delta_2(x,\eta_K)>0$ and $\ve(x,\eta_K)>0$ such that 
$\tau_x^{q_1} \nabla (d_p)_\ve (q_1) \in \partial d_p(x)_{\eta_K}$ and 
$\tau_x^{q_1}\nabla (d_q)_\ve (q_1) \in \partial d_q(x)_{\eta_K}$
for all $q_1\in B_{\delta_2(x,\,\eta_K)}(x)$ and $\ve\in(0,\ve(x,\eta_K))$. 
Since $\angle (\nabla ({d_p})_\ve (q_1), \nabla ({d_q})_\ve (q_1))
= \angle (\tau_x^{q_1} \nabla (d_p)_\ve (q_1), \tau_x^{q_1}\nabla (d_q)_\ve (q_1) )$, it follows from Lemma \ref{lemN3} that 
$\angle (\nabla ({d_p})_\ve (q_1), \nabla ({d_q})_\ve (q_1))> \pi/2$ 
for all $q_1\in B_{\delta_2(x,\,\eta_K)}(x)$ and all $\ve\in(0,\ve(x,\eta_K))$. 
This implies that, for any sufficiently small $\ve>0$, 
$\angle (\nabla (d_p)_\ve, \nabla (d_q)_\ve) > \pi/2$ holds on $K$ 
since $K$ is compact.$\qedd$
\end{proof}

\subsection{Approximations of Lipschitz maps: Proof of Theorem \ref{2015_07_30_theorem2}}\label{2015_08_26_Sect3.2}

In this subsection, using some techniques from 
non-smooth analysis and the partition of unity, 
we first treat an approximation of a Lipschitz map $F$
from a compact Riemannian manifold $M$ of dimension $n$ 
into a Riemannian manifold $N$ of dimension $k$, 
and finally prove Theorem \ref{2015_07_30_theorem2} 
applying the approximation method.\par 
Let $N$ be embedded into Euclidean $m$-dimensional space $(\R^m,\langle\,\cdot\,,\,\cdot\,\rangle)$, 
where $m \ge k+1$. 
We may assume that $N$ is isometrically 
embedded into $\R^m$ by introducing the induced metric from the space. 
Here, note that the notion of the singular point of a Lipschitz map is independent of 
the choice of the Riemannian metric (See \eqref{2015_08_02_Clarke_1} and Definition \ref{2014_02_20_def_1}). The Lipschitz map $F$ is therefore a map from $M$ into $\R^m$.
Then, we may define a smooth approximation of $F$ 
on a convex ball $B_r(p)$ of radius $r$, 
centered at each point $p \in M$.

\begin{definition}\label{defA2}
For each $\ve >0$, 
let $F^{(p)}_\ve: B_r(p) \lra \R^m$
denote the map defined by
\begin{equation}\label{eq:B1}
F^{(p)}_\ve (q):=\int _{\R^n} \rho_\ve (y) F(q(y)) dy=\int_{\R^n}\rho_\ve 
(\exp^{-1}_p q-y)F(\exp_p(y))dy,
\end{equation}
where $q(y):=\exp_p(\exp^{-1}_p q-y)$ and $\rho_\ve$ denotes the mollifier.
\end{definition}

\begin{lemma}\label{lemB2}
For any $\ve>0$ and any $q\in B_r(p)$, 
\[
\|F^{(p)}_\ve (q)-F(q)\|\le \ve \cdot \L(F)\cdot \L(\exp_p|_{\B_{r+\ve}(o_p)})
\]
holds, 
where $\|\,\cdot\, \|$ denotes the Euclidean norm of $\R^m$.
\end{lemma}

\begin{proof}
After the fashion of the proof of Lemma \ref{lemN5}, we have the desired inequality. 
$\qedd$
\end{proof}

Since $M$ is compact, we can choose  finitely many convex balls 
$B_{r_i}(p_i)$, $i=1,2, \dots, \ell$, which cover $M$. Take a partition of unity $\varphi_i$ subordinate to $\{ B_{r_i}(p_i) \}$. Then, for each $\ve>0$, we define the global approximation $F_\ve$ of $F$ by
\begin{equation}\label{eq:B2}
F_\ve(q)=\sum_{i=1}^{\ell} \varphi_i(q)F^{(p_i)}_\ve(q).
\end{equation}

\begin{lemma}\label{lemB3}
For any $\ve>0$ and any $q\in M$, we have 
\[
\|F_\ve(q)-F(q)\| \le \ve \cdot \L(F) \sum_{i=1}^\ell \varphi_i(q) \L(\exp_{p_i}|_{\B_{r_i+\ve}(o_{p_i})}).
\]
\end{lemma}

\begin{proof}
Since $\sum_{i=1}^\ell \varphi_i=1$, 
we get $F(q)=\sum_{i=1}^\ell\varphi_i(q)F(q)$. Hence, by Lemma \ref{lemB2} and the triangle inequality, we have the desired inequality.$\qedd$
\end{proof}

In what follows, for a pair of points $p$ and $q$ of $M$ or $N$ admitting a unique minimal geodesic segment $\gamma$, we denote by 
$\tau_q^p$ the parallel transport from the tangent space at $p$ onto the tangent space at $q$ along $\gamma$.

\begin{lemma}\label{lemB4}Let $q\in\supp\varphi_i$. 
Then, for any $\tilde u \in \Sph^{n-1}_q :=\{v\in T_qM\,|\, \|v\|=1\}$, 
\begin{equation}\label{eq:B4}
\|(dF_\ve^{(p_i)})_q(\tilde u)\| \le \L(F) 
\big(1+\sup_{y\in \B _{\ve}(o_{p_i}) }  \|Y_y^{(\tilde u)}(1)-\tau_{q_i(y)}^q(\tilde u)\| \big),
\end{equation}
holds, where $q_i(y):=\exp_{p_i}(\exp_{p_i}^{-1} q-y)$, 
and 
$
Y_y^{(\tilde u)}(t):=
\frac{\partial}{\partial s}
\exp_{p_i}t\left(\exp_{p_i}^{-1}(\exp_qs\tilde u)-y\right)\big|_{s = 0}
$ 
is a Jacobi field along the geodesic 
$\exp_{p_i} t\left(\exp_{p_i}^{-1} q-y\right)$ for each $y$.
\end{lemma}

\begin{proof}
From \eqref{eq:B1}, it is easy to obtain
\begin{equation}\label{eq:B5}
(dF_\ve^{(p_i)})_q (\tilde u)=\int _{\R^n}\rho_\ve(y)dF_{q_i(y)}(Y_y^{(\tilde u)}(1))dy
\end{equation}
and 
$
dF_{q_i(y)}(Y_y^{(\tilde u)}(1))
=
dF_{q_i(y)}(Y_y^{(\tilde u)}(1)-\tau^q_{q_i(y)}(\tilde u))
+dF_{q_i(y)}(\tau^q_{q_i(y)}(\tilde u))
$. Hence, by the triangle inequality, we get \eqref{eq:B4}.$\qedd$
\end{proof}

\begin{lemma}\label{lemB5}
For any $\eta>0$, there exists a number $\ve_i(\eta)>0$ such that
\[
\sup\{ \|Y_y^{(\tilde u)}(1)-\tau_{q_i(y)}^{q}(\tilde u)\| \,|\, q\in\supp\varphi_i,\,\tilde u\in \Sph_q^{n-1},\,y\in \B_{\ve_i(\eta)} (o_{p_i}) \} <\eta.
\]
\end{lemma}

\begin{proof}
In a similar way to the proof of Lemma \ref{lemN8}, we get the inequality.$\qedd$
\end{proof}

\begin{lemma}\label{lemB6}
For any $\eta>0$, there exists a number $\ve(\eta)>0$ such that
\begin{equation}\label{2014_08_18_lemB6_1}
\|dF_\ve(\tilde u)\| \le (1+\eta) \L(F)
\end{equation}
holds for all $\ve\in(0,\ve(\eta))$ and all unit tangent vectors $\tilde u$ on $M$.
\end{lemma}

\begin{proof}
Take any $\tilde{u} \in \Sph_q^{n-1}$, and fix it. 
Since $\sum_{i=1}^\ell \varphi_i=1$ on $M$, 
we get $\sum_{i=1}^\ell (d\varphi_i)_q(\tilde u)=0$. 
Then, we have 
\begin{equation}\label{2015_04_15_lemB6_1}
(dF_\ve)_q(\tilde u) 
=\sum_{i=1}^\ell \varphi_i(q) (dF_\ve^{(p_i)})_q(\tilde u)
+\sum_{i=1}^\ell (d\varphi_i)_q(\tilde u)(F_\ve^{(p_i)}(q)-F(q)).
\end{equation}
By applying the triangle inequality to the equation above,
we have 
\[
\|(dF_\ve)_q(\tilde u)\| 
\le 
\sum_{i=1}^\ell \varphi_i(q) \|(dF_{\ve}^{(p_i)})_q(\tilde u)\| 
+\sum_{i=1}^\ell |(d\varphi_i)_q(\tilde u)| \cdot \|F_\ve^{(p_i)}(q)-F(q)\|.
\] 
By Lemmas \ref{lemB2}, \ref{lemB4} and \ref{lemB5}, we get 
\eqref{2014_08_18_lemB6_1} for all sufficiently small $\ve>0$.
$\qedd$
\end{proof}

From now on, we assume that $n = \dim M \le \dim N = k$.

\begin{lemma}\label{lemB7}
For each non-singular point $p\in M$ of $F$, 
there exist positive numbers $r(p)$ and $\delta(p)$ 
such that, for any $u\in \Sph^{n-1}_p$, 
there exists a local unit vector field $V$ 
on a neighborhood of $F(p)$ satisfying 
\begin{equation}\label{eq:B8}
\langle dF_q(\tau_q^p(u)),V_{F(q)}\rangle\geq\delta(p)
\end{equation}
for almost all $q\in B_{2r(p)}(p)$. 
\end{lemma}

\begin{proof}
Choose convex balls $B_{r_1}(p)$ and $B_{r_2}(F(p))$ 
so as to satisfy
$F(B_{r_1}(p))\subset B_{r_2}(F(p))$. 
Since the point $p$ is a non-singular point of $F$, it follows from 
Example \ref{2015_07_25_Exa_3} that 
there exist positive numbers $r(p)\in(0,r_1/2)$ and $\delta(p)$ satisfying
the following property: For any $u \in \Sph^{n-1}_p$ at the point $p$, 
there exists a unit vector  $v$ at $F(p)$ such that 
\[
\langle\tau_{F(p)}^{F(q)}\circ dF_q\circ \tau^p_q(u),v\rangle\geq \delta(p)
\] 
for almost all $q\in B_{2r(p)}(p)$. Hence, we get
$\langle dF_q(\tau_q^p( u)),V_{F(q)}\rangle\geq \delta(p)$
for almost all $q\in B_{2r(p)}(p)$, where $V_{F(q)}:=\tau_{F(q)}^{F(p)}(v)$.
$\qedd$
\end{proof}

Henceforth, we fix any non-singular point $p \in M$ of $F$ and any $\tilde u\in \Sph_q^{n-1}$ at any point $q\in B_{r(p)}(p)$. 
Here, we also fix  an integer $i\in\{1,2, \dots,\ell\}$ satisfying 
$q\in \supp \varphi_i \subset B_{r_i}(p_i)$.  

\begin{lemma}\label{lemB8}
There exists a unit vector field $V$ on a neighborhood of $F(p)$ such that
\begin{align}\label{eq:B9}
&\langle (dF_\ve^{(p_i)})_q(\tilde u),V_{F(q)}\rangle\\[1mm]
&\ge -\L(F)\big(\sup_{y\in \B_\ve( o_{p_i}) } \|Y_y^{(\tilde u)}(1)-U_{q_i(y)}\|
+\sup_{y\in \B_\ve( o_{p_i})}\|V_{F(q)}-V_{F(q_i(y))}\|\big)
+\delta(p)\notag
\end{align}
for all $\ve\in(0,\ve_i(p))$. Here, 
$
\ve_i(p):=\min 
\left\{ 
r_i, r(p) / \L(\exp_{p_i}|_{\B_{2r_i}(o_{p_i})} )
\right\}
$ and 
$U_{q_i(y)} := \tau_{q_i(y)}^p\circ \tau_p^q (\tilde{u})$.
\end{lemma}

\begin{proof}
It follows from Lemma \ref{lemB7} that for the unit tangent vector $u:=\tau_p^q(\tilde u)$, 
there exists a unit vector field $V$ on a neighborhood of $F(p)$ satisfying \eqref{eq:B8}. 
By the triangle inequality, 
$\|\exp_{p_i}^{-1}q-y\|<r_i+\|y\|<2r_i$ for all $y\in \B_{r_i}(o_{p_i})$. 
Thus, we get
\[
d_M(q,q_i(y))\le \|y\| \cdot \L(\exp_{p_i}|_{\B_{2r_i}(o_{p_i})}  ) 
\]
for all $y \in \B_{r_i}(o_{p_i})$, 
where $d_M$ denotes the distance function of $M$.  
Then, from the triangle inequality, we obtain
\[
d_M(p,q_i(y))\leq d_M(p,q)+d_M(q,q_i(y))<2r(p)
\]
for all $y \in \B_{\ve_i (p)}(o_{p_i})$. 
Thus,  by Lemma \ref{lemB7},
\begin{equation}\label{eq:B10}
\langle dF_{q_i(y)}( U_{q_i(y)}  ),V_{F(q_i(y))}\rangle\geq\delta(p)
\end{equation}
for almost all $y \in \B_{\ve_i (p)}(o_{p_i})$, 
where $U_{q_i(y)}:=\tau^{p}_{q_i(y)}(u)$. 
It is clear to see that for almost all $y\in \B_{\ve}(o_{p_i})$,  
\begin{align}\label{eq:B11}
&\langle dF_{q_i(y)}(U_{q_i(y)}),V_{F(q)}\rangle\\[1mm] 
&=
\langle dF_{q_i(y)}( U_{q_i(y)}  ),V_{F(q)}-V_{F(q_i(y))}\rangle 
+ \langle dF_{q_i(y)}(U_{q_i(y)}),V_{F(q_i(y))}\rangle\notag\\[1mm]
&\ge-\L(F)\sup_{ y\in \B_{\ve}(o_{p_i})}\| V_{F(q)}-V_{F(q_i(y))} \| 
+\langle dF_{q_i(y)}( U_{q_i(y)}  ),V_{F(q_i(y))}\rangle.\notag
\end{align}
Since
$dF_{q_i(y)}(Y_y^{(\tilde u)}(1))=
dF_{q_i(y)}(Y_y^{(\tilde u)}(1)-U_{q_i(y)})+dF_{q_i(y)}(U_{q_i(y)})$, 
we therefore get \eqref{eq:B9} from \eqref{eq:B5}, \eqref{eq:B10} and \eqref{eq:B11}.$\qedd$
\end{proof}

\begin{lemma}\label{lemB9}
For any $\eta>0$, there exists a number $\ve_i(p,\eta) >0$
such that 
\[
\sup\{\|Y_y^{(\tilde u)}(1)-U_{q_i(y)}\|\,|\,  
q \in \supp\varphi_i \cap \ol{B_{r(p)}(p)},\, \tilde u\in \Sph_q^{n-1},\,y\in \B_{\ve_i(p,\,\eta)}(o_{p_i} )\}<\eta
\]
and
\[
\sup\{\|\tau_{F(q)}^{F(p)}(v) - \tau_{F(q_i(y))}^{F(p)}(v) \|\,|\, q\in\supp\varphi_i\cap\ol{B_{r(p)}(p)},\, 
v\in \Sph_{F(p)}^{k-1},\, y\in \B_{\ve_i(p,\,\eta)}(o_{p_i} )\}<\eta,
\]
where $\ol{B_{r(p)}(p)}:= \{q \in M\,|\, d(p, q) \le r(p)\}$. 
\end{lemma}

\begin{proof}
Imitate the proof of Lemma \ref{lemB5}.
$\qedd$
\end{proof}

\begin{lemma}\label{lemB10}
There exists $V_{F(q)} \in \Sph_{F(q)}^{k-1}$ such that
\[
\langle (dF_\ve^{(p_i)})_q(\tilde u),V_{F(q)}\rangle\ge \frac{2\delta(p)}{3}
\]
for all $\ve\in(0,\ve(p))$ and all $p_i$ with $q\in\supp\varphi_i$. 
Here we set 
\[
\ve(p)
:=
\min\{\ve_i(p),\ve_i(p,\eta_0)\,|\, \supp\varphi_i\cap B_{r(p)}(p)
\ne\emptyset\},
\]
where $\eta_0:=\delta(p) / 6\L(F)$.
\end{lemma}

\begin{proof}
The inequality is immediate from Lemmas \ref{lemB8} and \ref{lemB9}.
$\qedd$
\end{proof}

\begin{lemma}\label{lemB11}
There exist $V_{F(q)} \in \Sph_{F(q)}^{k-1}$ and a number $\ve_0(p)>0$ such that 
\begin{equation}\label{2014_08_24_lemB11_1}
\langle (dF_\ve)_q(\tilde u),V_{F(q)}\rangle \ge \frac{1}{3}\delta(p)
\end{equation}
for all $\ve\in(0,\ve_0(p))$.
\end{lemma}

\begin{proof}By Lemma \ref{lemB10}, 
for any $\ve\in(0,\ve(p))$,
\begin{equation}\label{eq:B13}
\sum_{i=1}^\ell \varphi_i(q)\langle dF_\ve^{(p_i)}(\tilde u),V_{F(q)}\rangle 
\ge \frac{2}{3}\delta(p).
\end{equation}
From Lemma \ref{lemB2}, we may choose a number 
$\ve_0(p)\in(0,\ve(p))$ satisfying
\begin{equation}\label{eq:B14}
\left|\sum_{i=1}^\ell d\varphi_i(\tilde u)\langle F_\ve^{(p_i)}(q)-F(q),V_{F(q)}\rangle \right| 
< \frac {1}{3}\delta(p)
\end{equation}
for all $\ve\in(0,\ve_0(p))$. 
Combining \eqref{2015_04_15_lemB6_1}, 
\eqref{eq:B13} and \eqref{eq:B14}, 
we get \eqref{2014_08_24_lemB11_1}.$\qedd$
\end{proof}

\bigskip\noindent
{\it (Proof of Theorem \ref{2015_07_30_theorem2})} 
Let $F: M \lra N$ be a Lipschitz map from a compact Riemannian manifold 
$M$ into a Riemannian manifold $N$, where $\dim M \le \dim N$. 
Assume that $F$ has no singular points on $M$. 
We embed $N$ into $\R^m$, where $m\ge \dim N +1$, 
and introduce the induced metric from the outer space to $N$. Note that $F:M\lra \R^m$. 
Since $N$ is a smooth submanifold in $\R^m$, 
for any $x\in N$ there exists an open neighborhood $U_x$ of $x$ in $\R^m$ such that each $y \in U_x$ admits 
a unique $z \in N$ with $\|y-z\|=\inf_{w\in N}\|y-w\|$. 
Considering the open set $\cU_N:=\cup_{x\in N}U_x$ in $\R^m$, we have the smooth locally distance projection 
$\pi_N:\cU_N \lra N$. Note here that $N \subset \cU_N$. 
Since $F(M)$ is compact, it follows from Lemma \ref{lemB3} that 
for all sufficiently small $\ve>0$, the image of the map $F_\ve$ 
defined by \eqref{eq:B2} is a subset of $\cU_N$. 
Then, for any sufficiently small $\ve>0$, 
we can define the smooth map $\pi_N \circ F_\ve$ from $M$ into $N$. By its definition, 
$(d\pi_N)_x$ is an orthogonal projection to $T_xN$ 
for each $x\in N$. Therefore, for any sufficiently small $\ve>0$, 
the map $\pi_N\circ F_\ve: M\lra N$ is an immersion from Lemma \ref{lemB11}. 
Combining Lemmas \ref{lemB3}, \ref{lemB6} and the argument above, 
we get the assertion.
$\qedd$

\subsection{Proof of Corollary \ref{2015_07_30_theorem1}}\label{sect3}

We need one more lemma in order to prove 
Corollary \ref{2015_07_30_theorem1}. 
In what follows, $d_M$ denotes the distance function of a given Riemannian manifold $M$.

\begin{lemma}\label{2014_02_05_lem1}
Let $M$ be a compact (connected) Riemannian manifold, and let 
$F:M\lra M$ be a local diffeomorphism. 
If $\max_{x \in M}d_M(F(x), x)$ 
is sufficiently small, then $F$ is injective. 
In particular, $F$ is a diffeomorphism.
\end{lemma}

\begin{proof}Since $F(M)$ is open and closed in $M$, 
$F$ is surjective. Then, we can consider $F$ to be a covering map. 
Since $M$ is compact, there exists a number $a$ such that 
$0 < a < i(M)$, where $i(M)$ denotes the injectivity radius of $M$. 
Since $\max_{x \in M}d_M(F(x), x)\ll 1$, we may assume that 
$d_M(F(x), x) < i(M)$ for all $x \in M$. Thus, for each $x \in M$, there exists a unique minimal geodesic segment $\gamma :[0,1]\lra M$ joining $x=\gamma(0)$ 
to $F(x)= \gamma(1)$. 
Define the map $H: M\times [0,1]\lra M$ by $H(x, t):= \sigma_x (t)$, 
where $\sigma_x :[0,1]\lra M$ denotes 
a unique minimal geodesic 
segment emanating from $x=\sigma_x (0)$ 
to $F(x)= \sigma_x (1)$. $H$ is continuous, for $F$ is continuous and $\sigma_x$ is unique. 
By the definition of $H$, $H(x, 0)=x= \id_M (x)$ and $H(x, 1)=F(x)$ for all $x \in M$, i.e., $F$ is homotopic to $\id_M$. 
Since $\pi_1 (M, F(x))/ F_\sharp(\pi_1 (M, x))$ 
is trivial, where $F_\sharp$ denotes 
the homomorphism from $\pi_1 (M, x)$ to $\pi_1 (M, F(x))$ 
induced by $F$, it follows from a well known lemma 
(cf.\,Corollary 3 in \cite[Chapter 3]{ST}) 
that $F^{-1} (x)$ is one point, i.e., $F$ is injective.
$\qedd$
\end{proof}

\bigskip\noindent
{\it (Proof of Corollary \ref{2015_07_30_theorem1})} 
Let $F$ be a bi-Lipschitz homeomorphism 
from a compact Riemannian manifold $M$ 
onto a Riemannian manifold $N$. 
Assume that $F$ and $F^{-1}$ have no singular points 
on $M$ and $N$, respectively. By Theorem \ref{2015_07_30_theorem2}, for any $\eta >0$, 
there exist two smooth immersions $f_\eta$ from $M$ 
into $N$ and $g_\eta$ from $N$ into $M$ such that 
\[
\max_{p\in M}d_N(f_\eta(p),F(p))<\eta, \quad \L(f_\eta)\le \L(F)(1+\eta),
\] 
and that 
\[
\max_{q\in N}d_M(g_\eta(q),F^{-1}(q))<\eta, \quad \L(g_\eta) \le \L(F^{-1})(1+\eta),
\]
respectively. Hence, $\L(g_\eta\circ f_\eta)\leq\L(g_\eta)\L(f_\eta)\leq\L(F)\L(F^{-1})(1+\eta)^2$.
Moreover, by the triangle inequality, we get
\[
d_M(g_\eta\circ f_\eta(p),p)
\le d_M(g_\eta(f_\eta(p)),F^{-1}(f_\eta(p))  )+d_M(F^{-1}(f_\eta(p) ),F^{-1}(F(p))).
\]
Since $d_M(g_\eta(f_\eta(p)),F^{-1}(f_\eta(p))   )<\eta$ 
and $d_M(F^{-1}(f_\eta(p)),F^{-1}(F(p)))<\eta\L(F^{-1})$ 
for all $p\in M$, we obtain 
$\max_{p\in M}d_M(g_\eta\circ f_\eta(p),p)<\eta (1+\L(F^{-1}))$.
Therefore, by Lemma \ref{2014_02_05_lem1}, 
for all sufficiently small $\eta>0$, $g_{\eta}\circ f_{\eta}$ 
is a diffeomorphism on $M$. 
This implies that $f_\eta$ and $g_\eta$ are injective 
for all sufficiently small $\eta$, and hence $M$ and $N$ are diffeomorphic. 
$\qedd$

\section{Proofs of the differentiable exotic sphere theorems}\label{sect2_new}

We first give two preliminaries to 
three cases \eqref{2015_07_25_MA_1}, 
\eqref{2015_08_12_MA_2}, and \eqref{2016_03_23_TAII_as}.  

\subsection{Preliminaries to the proofs in assuming \eqref{2015_07_25_MA_1} and \eqref{2016_03_23_TAII_as}}\label{sect2_new1}

Throughout this subsection, 
let $\sigma : S^{n-1}(1):=\{v \in \R^n\,|\, \|v\|=1\}\lra S^{n-1}(1)$ 
be a diffeomorphism satisfying 
\begin{equation}\label{2015_07_26_exotic}
\int^{\pi}_0 e^{-t} \| \ddot{c}(t) + c(t) \|dt \le e^{-\pi}\alpha 
\end{equation}
for some $\alpha >0$ and any unit speed geodesic $\gamma:[0,\pi]\lra S^{n-1}(1)$. 
Here, 
\[
c:= \sigma\circ \gamma, \quad \dot{c}:=\frac{dc}{dt}, 
\quad {\rm and} \quad \ddot{c}:=\frac{d^2c}{dt^2}.
\]
Our aim in this subsection is to prove the following theorem: 

\begin{theorem}\label{2015_07_25_theorem}
If $\alpha >0$ is sufficiently small, for example $\alpha = 1-1/\sqrt{2}$, then the origin 
$o \in \R^n$ is non-singular for the map 
$\wt{F}:\R^n\lra\R^n$ defined by 
\[
\wt{F}(v) := 
\begin{cases}
\ \|v\| \displaystyle{\sigma \left( \frac{v}{\|v\|} \right)}  \ &\text{on} \ \R^n \setminus\{o\},\\[2mm]
\ o  \ &\text{when} \ v=o.
\end{cases}
\]
\end{theorem}

We need six lemmas to prove Theorem \ref{2015_07_25_theorem}. 
So, we first prove the lemmas before proving the theorem: 
Take any unit speed geodesic $\gamma:[0,\pi]\lra S^{n-1}(1)$, and fix it. Set 
\[
\ol{c}(t):= c(0)\cos t + \dot{c}(0)\sin t
\]
on $[0,\pi]$. 

\begin{lemma}\label{2015_07_25_lemma1} 
For any $t \in [0,\pi]$, 
\[
\sqrt {\| c(t) - \ol{c}(t) \|^2 +  \| \dot{c}(t) - \dot{\ol{c}}(t) \|^2} 
\le e^\pi \int^{\pi}_0 e^{-\theta} \| \ddot{c}(\theta) + c(\theta) \|d\theta
\]
holds. In particular, we have 
\begin{equation}\label{2015_07_25_lemma1_0}
\| c(t) - \ol{c}(t) \| \le \alpha
\end{equation}
for all $t \in [0,\pi]$. 
\end{lemma}

\begin{proof}
Let $f(t):= \sqrt {\| c(t) - \ol{c}(t) \|^2 +  \| \dot{c}(t) - \dot{\ol{c}}(t) \|^2}$, and let   
$
X(t):=
\left(
\begin{array}{c}
c(t)\\[1mm]
\dot{c}(t)   
\end{array}
\right), \
\ol{X}(t):=
\left(
\begin{array}{c}
\ol{c}(t)\\[1mm]
\dot{\ol{c}}(t)   
\end{array}
\right) 
\in \R^{2n}$. 
Then, $f(t)= \|X(t)- \ol{X}(t)\|$. Note that $f(0)=0$ because 
$\ol{c}(0)= c(0)$ and $\dot{\ol{c}}(0)= \dot{c}(0)$. 
Choose an open interval $(a, b) \subset [0,\pi]$ such that $f(t) >0$ 
on $(a,b)$ and $f(a) =0$. In this situation, we observe   
\begin{align}\label{2015_07_25_lemma1_1}
f'(t)
&= \frac{1}{\|X(t)- \ol{X}(t)\|}\left\langle X(t)- \ol{X}(t), X'(t)- \ol{X}'(t)\right\rangle\\[1mm] 
&\le \|X'(t)- \ol{X}'(t)\|=\sqrt {\| \dot{c}(t) - \dot{\ol{c}}(t) \|^2 +  \| \ddot{c}(t) - \ddot{\ol{c}}(t) \|^2}\notag
\end{align}
on $(a,b)$. Since $\ddot{\ol{c}}(t)= - \ol{c}(t)$, we have 
\begin{align}\label{2015_07_25_lemma1_2}
\| \ddot{c}(t) - \ddot{\ol{c}}(t) \| 
&= \| (\ddot{c}(t) + c(t)) - (\ddot{\ol{c}}(t)+ c(t)) \|\\[1mm]
&\le \| \ddot{c}(t) + c(t)\| + \|\ddot{\ol{c}}(t)+ c(t) \|= \| \ddot{c}(t) + c(t)\| + \|c(t)-\ol{c}(t) \|.\notag
\end{align}
Thus, by \eqref{2015_07_25_lemma1_2}, 
\begin{align}\label{2015_07_25_lemma1_3}
&\| \dot{c}(t) - \dot{\ol{c}}(t) \|^2 +  \| \ddot{c}(t) - \ddot{\ol{c}}(t) \|^2\\[1mm]
&\le \| \dot{c}(t) - \dot{\ol{c}}(t) \|^2 
+ \| \ddot{c}(t) + c(t)\|^2 
+ 2 \| \ddot{c}(t) + c(t)\| \cdot \|c(t)-\ol{c}(t) \| +\|c(t)-\ol{c}(t) \|^2\notag\\[1mm]
&\le f^2(t) + 2 \| \ddot{c}(t) + c(t)\| f(t) + \| \ddot{c}(t) + c(t)\|^2 
= (f(t) + \| \ddot{c}(t) + c(t)\|)^2.\notag
\end{align}
Hence, by \eqref{2015_07_25_lemma1_1} and \eqref{2015_07_25_lemma1_3}, we get 
\begin{equation}\label{2015_07_25_lemma1_4}
f'(t) \le f(t) + \| \ddot{c}(t) + c(t)\|
\end{equation}
on $(a,b)$. Since $e^{-t} (f'(t) - f(t)) \le e^{-t} \| \ddot{c}(t) + c(t)\|$ from \eqref{2015_07_25_lemma1_4}, and since $f(a)=0$, 
\[
\int_a^b e^{-t} \| \ddot{c}(t) + c(t)\|dt \ge 
\int_a^b e^{-t} (f'(t) - f(t))dt =\int_a^b (e^{-t} f(t))'dt= e^{-b}f(b),
\]
and hence 
$
f(b) \le e^b\int_a^b e^{-\theta} \| \ddot{c}(\theta) + c(\theta)\|d\theta 
\le e^b\int_0^b e^{-\theta} \| \ddot{c}(\theta) + c(\theta)\|d\theta$. 
Since the function $t \lra e^t\int_0^t e^{-\theta} \| \ddot{c}(\theta) + c(\theta)\|d\theta$ is increasing on $[0,\pi]$, 
\begin{equation}\label{2015_07_25_lemma1_5}
f(b) \le e^b\int_0^b e^{-\theta} \| \ddot{c}(\theta) + c(\theta)\|d\theta 
\le e^\pi\int_0^\pi e^{-\theta} \| \ddot{c}(\theta) + c(\theta)\|d\theta.
\end{equation}
If $f(t)=0$ for some $t \in [0,\pi]$, then 
\eqref{2015_07_25_lemma1_5} still holds for such a $t$. 
Therefore, 
\[
f(t) \le e^\pi\int_0^\pi e^{-\theta} \| \ddot{c}(\theta) + c(\theta)\|d\theta.
\]
holds on $[0,\pi]$.$\qedd$
\end{proof}

\begin{lemma}\label{2015_07_25_lemma2}
$\langle c(0), c(\theta) \rangle \cos \theta \ge \cos^2\theta - \alpha |\cos \theta|$ for all $\theta \in [0, \pi]$.
\end{lemma}

\begin{proof}Since $\|c(0)\|=1$ and $c(0)\perp\dot{c}(0)$, 
\begin{align}\label{2015_07_25_lemma2_1}
\langle c(0), c(\theta) \rangle 
&=\langle c(0), c(\theta)-\ol{c}(\theta)\rangle + \langle c(0), \ol{c}(\theta) \rangle\\[1mm]
&=\langle c(0), c(\theta)-\ol{c}(\theta)\rangle + \langle c(0), c(0)\cos \theta 
+ \dot{c}(0)\sin \theta \rangle\notag\\[1mm]
&=\langle c(0), c(\theta)-\ol{c}(\theta)\rangle + \cos \theta.\notag
\end{align}
By \eqref{2015_07_25_lemma1_0} and \eqref{2015_07_25_lemma2_1}, we have 
\[
\langle c(0), c(\theta) \rangle \cos \theta 
\ge \cos^2 \theta - \|c(\theta)-\ol{c}(\theta)\| \cdot |\cos \theta|
\ge \cos^2 \theta - \alpha |\cos \theta|.
\]
$\qedd$
\end{proof}

\begin{lemma}\label{2015_07_25_lemma3}
$\langle \dot{c}(0), c(\theta) \rangle \sin \theta 
\ge \|\dot{c}(0)\|^2\sin^2\theta - \alpha\|\dot{c}(0)\| \sin \theta$ 
for all $\theta \in [0, \pi]$.
\end{lemma}

\begin{proof}Since $c(0)\perp\dot{c}(0)$, 
\begin{equation}\label{2015_07_25_lemma3_1}
\langle \dot{c}(0), c(\theta) \rangle 
=\langle \dot{c}(0), c(\theta)-\ol{c}(\theta)\rangle + \langle \dot{c}(0), \ol{c}(\theta) \rangle 
=\langle \dot{c}(0), c(\theta)-\ol{c}(\theta)\rangle + \|\dot{c}(0)\|^2\sin \theta.
\end{equation}
By \eqref{2015_07_25_lemma1_0} and \eqref{2015_07_25_lemma3_1}, we have 
\begin{align*}
\langle \dot{c}(0), c(\theta) \rangle \sin \theta 
&\ge \|\dot{c}(0)\|^2 \sin^2 \theta 
- \|\dot{c}(0)\| \cdot \|c(\theta)-\ol{c}(\theta)\| \sin \theta\\[1mm]
&\ge \|\dot{c}(0)\|^2\sin^2\theta - \alpha\|\dot{c}(0)\| \sin \theta.
\end{align*}
$\qedd$
\end{proof}

\begin{lemma}\label{2015_07_25_lemma4}
$|\|\dot{c}(0)\|-1| \le \alpha$.
\end{lemma}

\begin{proof}
Since $\ol{c}(\pi/2) = \dot{c}(0)$, $\dot{c}(0) = \ol{c}(\pi/2)- c(\pi/2) + c(\pi/2)$. 
Since $\|c(\pi/2)\|=1$, by the triangle inequality and \eqref{2015_07_25_lemma1_0}, 
$\|\dot{c}(0)\| \le \|\ol{c}(\pi/2)- c(\pi/2)\| + \|c(\pi/2)\| \le \alpha +1$, and 
$\|\dot{c}(0)\| \ge \|c(\pi/2)\| - \|\ol{c}(\pi/2)- c(\pi/2)\| \ge 1-\alpha$. Hence, we get the assertion. 
$\qedd$
\end{proof}

\begin{lemma}\label{2015_07_25_lemma5}
For $\alpha =1- 1/\sqrt{2}$, 
we have 
\[
\langle \ol{c}(\theta), c(\theta)\rangle >0
\]
on $[0, \pi]$.
\end{lemma}

\begin{proof}
First we consider the case where $\|\dot{c}(0)\|\ge1$ holds. 
Then, by Lemmas \ref{2015_07_25_lemma2}, \ref{2015_07_25_lemma3}, and \ref{2015_07_25_lemma4}, 
we have 
\begin{align}\label{2015_07_25_lemma5_1}
\langle \ol{c}(\theta), c(\theta)\rangle
&=\langle c(0), c(\theta) \rangle \cos \theta 
+ \langle \dot{c}(0), c(\theta) \rangle \sin \theta\\[1mm] 
&\ge
\cos^2\theta - \alpha |\cos \theta| 
+ \|\dot{c}(0)\|^2\sin^2\theta - \alpha\|\dot{c}(0)\| \sin \theta\notag\\[1mm]
&\ge \cos^2\theta - \alpha |\cos \theta| 
+ \sin^2\theta - \alpha(\alpha +1) \sin \theta\notag\\[1mm]
&= 1 - \alpha |\cos \theta| - \alpha(\alpha +1) \sin \theta.\notag
\end{align}
Set $t = \sin \theta$. Note that $0\le t\le 1$. 
Since 
$|\cos \theta|= \sqrt{1-t^2}\le 1- t^2/2$, 
we have, 
by \eqref{2015_07_25_lemma5_1}, 
\begin{align}\label{2016_03_22_lemma5_new}
\langle \ol{c}(\theta), c(\theta)\rangle
&\ge 1 - \alpha |\cos \theta| - \alpha(\alpha +1) \sin \theta\\[1mm]
&\ge 1 +\alpha \left(\frac{t^2}{2} -1\right)-\alpha(\alpha +1)t\notag
\\[1mm]
&=\frac{\alpha}{2}t^2-\alpha(\alpha +1)t +1-\alpha\notag\\[1mm]
&=\frac{\alpha}{2}\left\{t -(\alpha +1)\right\}^2
-\frac{\alpha}{2}(\alpha +1)^2 +1-\alpha.\label{2016_03_22_lemma5_new2}
\end{align}
Consider the function $\varphi(t):=(\alpha/2) t^2-\alpha(\alpha +1)t +1-\alpha$ on $[0,1]$, which is a parabola. 
From \eqref{2016_03_22_lemma5_new2}, the axis of $\varphi$ 
is $t=\alpha+1$. Since $\alpha =1- 1/\sqrt{2}$, 
we see that $\alpha/2>0$ and $\alpha+1>1$, 
which imply that $\varphi$ is decreasing on $[0,1]$. 
Moreover, since $\alpha =1- 1/\sqrt{2}$, we see that 
\begin{equation}\label{2016_03_22_lemma5_new3}
\varphi (1)= -\left(\alpha +\frac{3}{4}\right)^2+\frac{25}{16}>0. 
\end{equation}
Hence, by \eqref{2016_03_22_lemma5_new}, \eqref{2016_03_22_lemma5_new2}, and \eqref{2016_03_22_lemma5_new3}, 
\[
\langle \ol{c}(\theta), c(\theta)\rangle \ge \varphi (1) >0
\]
holds on $[0, \pi]$.\par 
Finally we consider the case where $\|\dot{c}(0)\|<1$ holds. 
Then, since $\alpha =1- 1/\sqrt{2}$, 
we have, by Lemmas \ref{2015_07_25_lemma2}, \ref{2015_07_25_lemma3}, and \ref{2015_07_25_lemma4}, 
\begin{align*}
\langle \ol{c}(\theta), c(\theta)\rangle
&=\langle c(0), c(\theta) \rangle \cos \theta 
+ \langle \dot{c}(0), c(\theta) \rangle \sin \theta\\[1mm] 
&\ge
\cos^2\theta - \alpha |\cos \theta| 
+ \|\dot{c}(0)\|^2\sin^2\theta - \alpha\|\dot{c}(0)\| \sin \theta\\[1mm]
&> 
\cos^2\theta + \sin^2\theta - \alpha |\cos \theta| 
+ (\|\dot{c}(0)\|^2-1)\sin^2\theta - \alpha\sin \theta\notag\\[1mm]
&= 1 - \alpha (|\cos \theta|+\sin \theta) + (\|\dot{c}(0)\|+1)(\|\dot{c}(0)\|-1)\sin^2\theta\\[1mm]
&\ge 1 - \alpha (|\cos \theta|+\sin \theta) -\alpha (\|\dot{c}(0)\|+1)\sin^2\theta\\[1mm]
&> 1 - \sqrt{2}\alpha -2\alpha = 0
\end{align*}
for all $\theta \in [0,\pi]$.
$\qedd$
\end{proof}

\begin{lemma}\label{2015_07_26_lemma7}
Let $\wt{F}:\R^n\lra\R^n$ be the map defined by 
\begin{equation}\label{2015_07_26_lemma7_1}
\wt{F}(v) := 
\begin{cases}
\ \|v\| \displaystyle{\sigma \left( \frac{v}{\|v\|} \right)}  \ &\text{on} \ \R^n \setminus\{o\},\\[2mm]
\ o  \ &\text{when} \ v=o.
\end{cases}
\end{equation}
Then, the generalized differential of $\wt{F}$ at $o$ is given by 
\[
\partial \wt{F} (o)= \Conv (\{A_v \,|\, v \in S^{n-1}(1)\}),
\]
where $A_v$ is the linear map defined by $A_v (\lambda v):= \lambda \sigma (v)$ for all $\lambda \in \R$ and $A_v (w):= d\sigma_v (w)$ for 
all $w\in \R^n$ with $w\perp v$. 
\end{lemma}

\begin{proof} Let $v \in S^{n-1}(1)$, $\ell>0$, 
and $\lambda \in \R$. 
By \eqref{2015_07_26_lemma7_1}, 
\begin{equation}\label{2015_07_26_lemma7_2}
d\wt{F}_{\ell v} (\lambda v) 
= \frac{d}{dt} \wt{F}(\ell v + t\lambda v)\bigg|_{t=0}
= \frac{d}{dt} (\ell + \lambda t) \sigma (v)\bigg|_{t=0}
= \lambda\sigma (v).
\end{equation}
For any $w\in \R^n$ with $w\perp v$, by \eqref{2015_07_26_lemma7_1}, 
\begin{align}\label{2015_07_26_lemma7_3}
d\wt{F}_{\ell v} (w)
= \frac{d}{dt} \wt{F}(\ell v + tw)\bigg|_{t=0} 
&= \frac{d}{dt} \wt{F}\left( \sqrt{\ell^2 + \|w\|^2t^2}\cdot \frac{\ell v + tw}{\sqrt{\ell^2 + \|w\|^2t^2}}
\right)\bigg|_{t=0}\\[1mm]
&= \frac{d}{dt} \sqrt{\ell^2 + \|w\|^2t^2}\cdot\sigma\left( \frac{\ell v + tw}{\sqrt{\ell^2 + \|w\|^2t^2}}
\right)\bigg|_{t=0}\notag\\[1mm]
&= \ell \cdot \frac{d}{dt}\sigma\left( \frac{\ell v + tw}{\sqrt{\ell^2 + \|w\|^2t^2}}
\right)\bigg|_{t=0}\notag\\[1mm]
&= \ell \cdot d\sigma_v \left(\frac{w}{\ell}\right)= d\sigma_v \left(w\right)\notag
\end{align}
Hence, by \eqref{2015_07_26_lemma7_2} and \eqref{2015_07_26_lemma7_3}, 
we obtain $\partial \wt{F} (o)= \Conv (\{A_v \,|\, v \in S^{n-1}(1)\})$.
$\qedd$
\end{proof}

\begin{remark} In the proof of Lemma \ref{2015_07_26_lemma7}, 
we do not need the assumption that $\sigma$ satisfies \eqref{2015_07_26_exotic} for some $\alpha$. 
\end{remark}

\medskip

Now, we are going to give the proof 
of Theorem \ref{2015_07_25_theorem}:

\medskip\noindent
{\it (Proof of Theorem \ref{2015_07_25_theorem})} 
Let $\alpha >0$ be sufficiently small, and let 
$\wt{F}:\R^n\lra\R^n$ be the map defined 
by \eqref{2015_07_26_lemma7_1}. Let $u,v\in S^{n-1}(1)$ be any vectors. Then, we can take a unit speed geodesic 
$\gamma :[0, \pi]\lra S^{n-1}(1)$ emanating from $v= \gamma (0)$ satisfying $u= \gamma (t_0)$ for some $t_0\in [0,\pi]$. 
Setting $w:= \dot{\gamma} (0)$, we have 
$\gamma (t)= v \cos t + w \sin t$ on $[0,\pi]$. 
Take any $\ell >0$, and fix it. 
Since $c(0)= \sigma (\gamma(0))= \sigma (v)$ and 
$\dot{c}(0)= d\sigma_{\gamma(0)}(\dot{\gamma}(0)) = d\sigma_{v}(w)$, 
we see, 
by \eqref{2015_07_26_lemma7_2} and \eqref{2015_07_26_lemma7_3}, 
\begin{align*}
d\wt{F}_{\ell v} (u) = d\wt{F}_{\ell  v} (\gamma (t_0))
&= d\wt{F}_{\ell v}(v)\cos t_0 
+ d\wt{F}_{\ell v}(w) \sin t_0\\[1mm]
&= \sigma (v)\cos t_0 + d\sigma_v(w) \sin t_0\\[1mm]
&= c(0)\cos t_0 + \dot{c}(0) \sin t_0 = \ol{c}(t_0).
\end{align*}
Since $\sigma (u) = \sigma (\gamma (t_0)) = c(t_0)$, 
\begin{equation}\label{2016_03_22_theorem_new}
\langle d\wt{F}_{\ell v} (u), \sigma (u) \rangle 
= \langle d\wt{F}_{\ell v} (u), c(t_0) \rangle
= \langle \ol{c}(t_0), c(t_0) \rangle.  
\end{equation}
Thus, if $\alpha= 1-1/\sqrt{2}$, 
then, by Lemma \ref{2015_07_25_lemma5}, 
$\langle d\wt{F}_{\ell v} (u), \sigma (u) \rangle >0$ holds. This implies
\begin{equation}\label{2015_07_25_theorem_1}
\langle A_{v} (u), \sigma (u) \rangle >0
\end{equation}
for all $u, v \in S^{n-1}(1)$. 
Therefore, by Lemma \ref{2015_07_26_lemma7} and \eqref{2015_07_25_theorem_1}, $o$ is non-singular for $\wt{F}$. 
Indeed, suppose that $o$ is singular for $\wt{F}$. 
Set $G:= \{A_v \,|\, v \in S^{n-1}(1)\}$. Since 
\[
\Conv (G) = \bigg\{\sum_{i=1}^{\ell}\lambda_iA_{z_i}\,\bigg|\, 
A_{z_i} \in G, \sum_{i=1}^{\ell}\lambda_i =1, \lambda_i\ge 0 \,(i=1,2,\ldots, \ell)\bigg\},
\]
where $\ell < \infty$ (by Carath\'eodory's theorem), there exists $\sum_{i=1}^{\ell}\lambda_iA_{z_i} \in \partial \wt{F}(o)$ such that 
$\mathrm{rank} (\sum_{i=1}^{\ell}\lambda_iA_{z_i}) < n$. Thus, we may find a vector $v_0 \in S^{n-1}(1)$ such that $\sum_{i=1}^{\ell}\lambda_iA_{z_i} (v_0) =o$. 
By \eqref{2015_07_25_theorem_1}, 
\[
0= \langle o, \sigma (v_0) \rangle 
= \left\langle \sum_{i=1}^{\ell}\lambda_iA_{z_i} (v_0), \sigma (v_0) \right\rangle 
= \sum_{i=1}^{\ell}\lambda_i \langle A_{z_i} (v_0), \sigma (v_0) \rangle >0,
\]
which is a contradiction. Hence, $o$ is non-singular for $\wt{F}$.
$\qedd$

\bigskip

From \eqref{2016_03_22_theorem_new} 
in the proof of Theorem \ref{2015_07_25_theorem}, 
we have the following lemma, 
which is applied to the proof of the differentiable exotic 
sphere theorem II.

\begin{lemma}\label{2016_03_23_lem1} If $\sigma$ satisfies 
\[
\angle(\ol{c}(t), c(t)) < \frac{\pi}{2}
\]
for all geodesic segments 
$\gamma ([0,\pi])\subset S^{n-1}(1)$ with $\|\dot{\gamma}\|\equiv 1$, 
then $o \in \R^n$ is non-singular 
for the map $\wt{F}$ defined by \eqref{2015_07_26_lemma7_1}. 
Here, we do not assume that 
$\sigma$ satisfies \eqref{2015_07_26_exotic} for some $\alpha$. 
\end{lemma}

We need a corollary below 
of Theorem \ref{2015_07_25_theorem} to prove the differentiable exotic 
sphere theorem I 
in the case where \eqref{2015_07_25_MA_1} holds.
We prepare the following lemma for proving the corollary.

\begin{lemma}\label{2015_07_26_lemma8} A smooth curve $c:(a,b)\lra S^{n-1}(1)\subset \R^n$ satisfies
\[
\|\ddot{c}(t)+ c(t)\|^2= \|\ddot{c}(t)\|^2- 2\|\dot{c}(t)\|^2+1.
\]
\end{lemma}

\begin{proof} It is clear that $\|\ddot{c}(t)+ c(t)\|^2 
= \|\ddot{c}\|^2 +2\langle \ddot{c}(t), c(t) \rangle+1$, where note that $\|c(t)\| \equiv1$ on $(a,b)$. 
Since $1= \|c(t)\|^2= \langle c(t), c(t) \rangle$ 
for all $t \in (a,b)$, 
$\langle \dot{c}(t), c(t) \rangle\equiv 0$ holds. We thus 
have 
$\langle \ddot{c}(t), c(t) \rangle= - \langle \dot{c}(t), \dot{c}(t) \rangle = -\|\dot{c}(t)\|^2$. 
Hence, we get 
\[
\|\ddot{c}(t)+ c(t)\|^2
= \|\ddot{c}\|^2 +2\langle \ddot{c}(t), c(t) \rangle+1
= \|\ddot{c}(t)\|^2- 2\|\dot{c}(t)\|^2+1.
\]
$\qedd$
\end{proof}

\begin{corollary}\label{2015_08_31_cor}
Let $\sigma : S^{n-1}(1)\lra S^{n-1}(1)$ be a diffeomorphism, 
$c:=\sigma\circ\gamma$ the curve in $S^{n-1}(1)$, 
where $\gamma:[0,\pi]\lra S^{n-1}(1)$ denotes a geodesic segment, 
and $\wt{F}:\R^n\lra\R^n$ the map defined 
by \eqref{2015_07_26_lemma7_1}. 
If $\bL(\sigma)$ satisfies \eqref{2015_07_25_MA_1} 
for all geodesic segments 
$\gamma ([0,\pi])\subset S^{n-1}(1)$ with $\|\dot{\gamma}\|\equiv 1$, 
then $o \in \R^n$ is non-singular for $\wt{F}$.
\end{corollary}

\begin{proof}Since 
\[
\|\dot{c}(t)\|
= \lim_{h\to0}\frac{\|c(t+h) - c(t)\|}{|h|}
= \lim_{h\to0}
\frac{\|\sigma(\gamma(t+h)) - \sigma (\gamma(t))\|}
{\|\gamma(t+h) - \gamma(t)\|}
\cdot \frac{\|\gamma(t+h) - \gamma(t)\|}{|h|},
\]
\begin{equation}\label{2015_08_31_cor1}
\bL(\sigma)^{-1}\leq \|\dot c\|\leq \bL(\sigma)
\end{equation}
holds on any geodesic segment $\gamma ([0,\pi])\subset S^{n-1}(1)$. 
By Lemma \ref{2015_07_26_lemma8}, \eqref{2015_07_25_MA_1}, 
and \eqref{2015_08_31_cor1}, 
\begin{align*}
\int_0^\pi e^{-t}\|\ddot{c}(t)+ c(t)\|\,dt 
&= \int_0^\pi e^{-t}
\sqrt{\|\ddot{c}(t)\|^2- 2\|\dot{c}(t)\|^2+1}\,dt \\[1mm]
&\le \int_0^\pi e^{-t} 
\sqrt{\|\ddot{c}(t)\|^2- 2\bL(\sigma)^{-2} +1}\,dt\\[1mm]
&\le \int_0^\pi e^{-t}dt\, 
\sqrt{2 \left\{ \frac{\sqrt{2} -1 }{2(e^\pi -1)} \right\}^2} \\[1mm]
&= (1-e^{-\pi}) \cdot \frac{2- \sqrt{2}}{2(e^\pi-1)} \\[1mm]
&= e^{-\pi}\left(1- \frac{1}{\sqrt{2}}\right).
\end{align*}
Hence, $\sigma$ satisfies \eqref{2015_07_26_exotic} 
for $\alpha=1-1/\sqrt{2}$. 
By Theorem \ref{2015_07_25_theorem}, 
$o$ is non-singular for $\wt{F}$.$\qedd$
\end{proof}

\subsection{Preliminary to the proof 
in assuming \eqref{2015_08_12_MA_2}}\label{sect2_new2}

\begin{lemma}\label{2015_08_13_lem_second}
Let $\sigma: S^{n-1}(1)\lra S^{n-1}(1)$ be 
a diffeomorphism, 
where we do not assume \eqref{2015_07_26_exotic}, 
$\bL(\sigma)$ the bi-Lipschitz constant of it, 
and $\wt{F}:\R^n\lra\R^n$ the map defined 
by \eqref{2015_07_26_lemma7_1}. Then, we have 
\begin{equation}\label{2015_08_13_lem_second_1}
\bL(\sigma)^{-1}\|u-v\| \le \|\wt{F}(u)-\wt{F}(v)\|\le \bL(\sigma)\|u-v\|
\end{equation}
for all $u,v \in \R^n$, i.e., $\wt{F}$ is bi-Lipschitz.
\end{lemma}

\begin{proof}Take any $u, v \in \R^n$. If $u=0$, or $v=0$, 
then the inequality holds. Hence, we assume $u\not=0$ and $v\not=0$. 
Set $\tilde{u}:= u/\|u\|$ and $\tilde{v}:= v/\|v\|$. 
Note that if $\tilde{u}= \tilde{v}$, then \eqref{2015_08_13_lem_second_1} trivially 
holds. Hence, we assume $\tilde{u}\not= \tilde{v}$. 
Since $\|\wt{F}(u)\| = \|u\|$ and 
$
\langle \wt{F} (u), \wt{F}(v) \rangle 
= \|u\|\cdot\|v\|\langle \sigma(\tilde{u}), \sigma(\tilde{v}) \rangle
$, 
\begin{equation}\label{2015_08_28_Lem2.10_1}
\| \wt{F} (u)- \wt{F}(v) \|^2
= \|u\|^2 + \|v\|^2 -2\|u\|\cdot\|v\| 
\langle \sigma(\tilde{u}),\sigma(\tilde{v}) \rangle
\end{equation}
Since $\|u-v\|^2= \|u\|^2 + \|v\|^2 -2\|u\|\cdot\|v\| 
\langle \tilde{u},\tilde{v} \rangle$, 
we have, by \eqref{2015_08_28_Lem2.10_1}, 
\begin{equation}\label{2015_08_28_Lem2.10_2}
\| \wt{F} (u)- \wt{F}(v) \|^2 - \|u-v\|^2
= 2\|u\|\cdot\|v\| 
(\langle \tilde{u},\tilde{v} \rangle 
- \langle \sigma(\tilde{u}),\sigma(\tilde{v}) \rangle)
\end{equation}
Similarly, we see 
\begin{equation}\label{2015_08_28_Lem2.10_3}
\| \sigma (\tilde{u})- \sigma(\tilde{v}) \|^2 - \|\tilde{u}-\tilde{v}\|^2
= 2(\langle \tilde{u},\tilde{v} \rangle 
- \langle \sigma(\tilde{u}),\sigma(\tilde{v}) \rangle)
\end{equation}
Set $\ell^2:=\| \sigma (\tilde{u})- \sigma(\tilde{v}) \|^2/\|\tilde{u}-\tilde{v}\|^2$. 
Then, by \eqref{2015_08_28_Lem2.10_2} and \eqref{2015_08_28_Lem2.10_3}, 
\begin{align}\label{2015_08_28_Lem2.10_4}
\|\wt{F} (u)-\wt{F}(v) \|^2 -\| u - v \|^2
&= \|u\|\cdot\|v\| 
\{\| \sigma (\tilde{u})- \sigma(\tilde{v}) \|^2 - \|\tilde{u}-\tilde{v}\|^2\}\\[1mm]
&= \|u\|\cdot\|v\|(\ell^2-1)\|\tilde{u}-\tilde{v}\|^2\notag\\[1mm]
&=2 (\ell^2-1)(\|u\|\cdot\|v\|-\langle u, v\rangle).\notag
\end{align}
Since $\bL(\sigma)\ge1$, $\|u\|\cdot\|v\|-\langle u, v\rangle \ge0$ and $2(\|u\|\cdot\|v\| - \langle u,v\rangle)\le \|u-v\|^2$, we see, 
by \eqref{2015_08_28_Lem2.10_4}, 
\begin{align*}
\|\wt{F} (u)- \wt{F}(v) \|^2 
&= \| u - v \|^2 +2 (\ell^2-1)(\|u\|\cdot\|v\|-\langle u, v\rangle)\\[1mm]
&\le \| u - v \|^2 +
2 (\bL(\sigma)^2-1)(\|u\|\cdot\|v\|-\langle u, v\rangle)\\[1mm]
&= \bL(\sigma)^2 \| u - v \|^2 +(1-\bL(\sigma)^2)
\{
\| u - v \|^2 -2(\|u\|\cdot\|v\|-\langle u, v\rangle)
\}\\[1mm]
&\le \bL(\sigma)^2 \| u - v \|^2.
\end{align*}
As well as above, we have 
$\|\wt{F} (u)- \wt{F}(v) \|^2 \ge \bL(\sigma)^{-2} \| u - v \|^2$.$\qedd$
\end{proof}

\subsection{The proofs}\label{sect2_new3}

All notations in this subsection are the same 
as those defined in Sect.\,\ref{2016_01_11_new}.

\bigskip\noindent
{\it (Proof of the differentiable exotic sphere 
theorem I in the case of \eqref{2015_07_25_MA_1})} 
We assume here $d_{M_i}(p_i,q_i):= \pi$ for each $i=1,2$ 
by normalizing the metric, where $d_{M_i}$ denotes 
the distance function on $M_i$. 
Note that this normalization 
does not change $\bL(\sigma)$ and $\|\ddot{c}\|$. 
The diffeomorphism 
$\sigma^{p_i}_{q_i} : \Sph^{n-1}_{p_i}\lra \Sph^{n-1}_{q_i}$ defined by 
\eqref{2015_07_30_def_of_subsigma} satisfies 
$\sigma^{p_i}_{q_i} (u_i)= -\dot{\tau}_{u_i} (\pi)$, i.e., $\ell_i=\pi$, 
since $d_{M_i}(p_i,q_i)= \pi$. Here, $\tau_{u_i}(t):= \exp_{p_i}tu_i$ for all $u_i \in \Sph^{n-1}_{p_i}$ and all $t \in [0,\pi]$.\par 
We first construct a map $\wt{F}$ 
satisfying \eqref{2015_07_26_lemma7_1} 
under identifying $T_{q_i}M_i = \R^n$:  
For each $(t, u_1) \in [0,\pi] \times \Sph^{n-1}_{p_1}$, 
we define the bi-Lipschitz homeomorphism $F: M_1\lra M_2$ by 
\begin{equation}\label{2015_08_24_new1}
F(\exp_{p_1}tu_1):=\exp_{p_2}(t I(u_1)), 
\end{equation}
where $I: T_{p_1}M_1\lra T_{p_2}M_2$ denotes a linear isometry. 
Note that $F$ is a diffeomorphism between $M_1\setminus\{q_1\}$ and 
$M_2\setminus\{q_2\}$. 
Define the map $\wt{F}: \B_\pi(o_{q_1})\lra \B_\pi(o_{q_2})$ by 
\begin{equation}\label{2015_08_24_new2}
\wt{F}:= \exp_{q_2}^{-1}\circ F \circ \exp_{q_1},
\end{equation}
where for each $i=1,2$, $\B_\pi (o_{q_i}):=\{v \in T_{q_i}M_i\,|\, \|v\|<\pi\}$ and $o_{q_i}$ denotes the origin of $T_{q_i}M_i$. Since 
\[
F(\exp_{q_1}t \sigma^{p_1}_{q_1}(u_1))
= F(\tau_{u_1}(\pi-t)) 
= \exp_{p_2}(\pi-t)I(u_1)
= \tau_{u_2}(\pi-t)
= \exp_{q_2}t \sigma^{p_2}_{q_2}(u_2),
\] 
where $u_2:= I(u_1)$, we have 
\begin{equation}\label{2015_07_26_theorem1_1} 
\wt{F}(t \sigma^{p_1}_{q_1}(u_1)) 
= \exp_{q_2}^{-1}(\exp_{q_2}t \sigma^{p_2}_{q_2}(u_2))
= t \sigma^{p_2}_{q_2}(u_2) = t \sigma^{p_2}_{q_2}\circ I(u_1)
\end{equation}
By setting $v_1:= \sigma^{p_1}_{q_1}(u_1)$, 
it follows from \eqref{2015_07_26_theorem1_1} that 
\[
\wt{F}(t v_1) = \wt{F}(t \sigma^{p_1}_{q_1}(u_1))  
= t \sigma^{p_2}_{q_2}\circ I(u_1)
= t \sigma^{p_2}_{q_2}\circ I \circ \sigma^{q_1}_{p_1}(v_1) = t \sigma (v_1), 
\]
and hence $\wt{F}$ 
satisfies \eqref{2015_07_26_lemma7_1}.\par  
Since 
$d(\exp_{q_1})_{o}$ and $d(\exp_{q_2}^{-1})_{q_2}$ 
are the identity maps, 
$\partial F(q_1)= \partial \wt{F}(o)$, 
where $o:=o_{q_1}$. By Lemma \ref{2015_07_26_lemma7}, 
$
\partial F(q_1)
= \Conv (\{A_v \,|\, v \in \Sph^{n-1}_{q_1}\})
$ 
holds, where $A_v$ is as in the statement. 
Since $\bL(\sigma)$ satisfies \eqref{2015_07_25_MA_1}, 
by Corollary \ref{2015_08_31_cor}, 
$q_1$ is non-singular for $F$. 
Since $F$ has no singular points on $M_1$, 
it follows from 
Theorem \ref{2015_07_30_theorem2} that 
for any sufficiently small $\eta>0$,  
there exists a smooth approximation 
$f_\eta : M_1\lra M_2$ of $F$, which is an 
immersion. 
Since $f_\eta (M_1)$ is open and closed in $M_2$, 
$f_\eta (M_1)=M_2$. On the other hand, since $M_1$ is compact and $M_2$ 
is Hausdorff, $f_\eta$ is proper. Since $f_\eta$ is a local 
diffeomorphism, $f_\eta$ is a covering 
map from $M_1$ onto $M_2$. 
Since $M_2$ is simply connected, $f_\eta$ is injective. 
Therefore, for any sufficiently small $\eta >0$, 
$f_\eta$ is a diffeomorphism from $M_1$ onto $M_2$.
$\qedd$

\bigskip\noindent
{\it (Proof of the differentiable exotic 
sphere theorem I in the case of \eqref{2015_08_12_MA_2})} 
Let $F:M_1\lra M_2$ and $\wt{F}:\B_\pi (o_{q_1})\lra \B_\pi (o_{q_2})$ 
be the maps defined by \eqref{2015_08_24_new1} and \eqref{2015_08_24_new2}, respectively. 
Moreover, let $M_2$ be isometrically embedded into $\R^m$, 
where $m\ge n+1$, by introducing the induced metric from the space. 
Then, $F:M_1\lra \R^m$. By Lemma \ref{2015_08_13_lem_second} and 
\eqref{2015_08_12_MA_2}, we get 
$\bL(\wt{F})^2 \le 1+ \left\{ (8/\pi)(n-1)\right\}^{-1/2}$. Hence, 
it follows from the proof of \cite[Theorem 5.1]{K} that 
the smooth approximation $\wt{F}_\ve$, which is the convolution 
of $\wt{F}$ and the mollifier $\rho_\ve$, is an immersion 
from some open ball $\B_\delta (o_{q_1})\subset \B_\pi (o_{q_1})$ 
into $\B_\pi (o_{q_2})$. 
Here, $\wt{F}_\ve(y):=\int \wt{F}(x) \rho_\ve (x-y)dx$ for $\ve<\delta$. 
Therefore, $F$ admits a local smooth approximation $F_\ve^{(q_1)}$ 
from an open ball $B_\delta (q_1)=\exp_{q_1}\B_\delta (o_{q_1})$ into $\R^m$, which is an immersion. 
Let $\varphi:M_1\lra \R$ be a smooth function satisfying 
$0\le \varphi \le1$ on $M_1$, $\varphi\equiv 1$ on $\ol{B_r(q_1)}$, 
and $\supp \varphi\subset B_R(q_1)$, where 
$0<r < R<\delta$. 
Define the map $F_\ve:M_1\lra \R^m$ 
by $F_\ve:= (1-\varphi)F + \varphi F_\ve^{(q_1)}$. Since 
\[
F_\ve = 
\begin{cases}
\ F_\ve^{(q_1)} \ &\text{on} \ \ol{B_r(q_1)},\\[1mm]
\ F  \ &\text{on} \ M_1\setminus \supp \varphi, 
\end{cases}
\]
$F_\ve$ is a local diffeomorphism from 
$B_r(q_1)\cup (M_1\setminus \supp \varphi)$ into $\R^m$. 
Since $F$ is smooth on $B_R(q_1) \setminus \ol{B_r(q_1)}$, 
it is easy to see that $F_\ve^{(q_1)}$ uniformly converges to $F$ 
on $\ol{B_R(q_1)} \setminus B_r(q_1)$ as $\ve$ goes to zero in the $C^1$-topology. Thus, $F_\ve$ uniformly converges to $F$ on $\ol{B_R(q_1)} \setminus B_r(q_1)$ as $\ve$ goes to zero in the $C^1$-topology. Now, 
choose a small neighborhood $\cN_{M_2}$ of $M_2\subset \R^m$ such that any $x \in \cN_{M_2}$ admits a unique nearest point 
$\pi_{M_2}(x)\in M_2$. Then, the map $\pi_{M_2} : \cN_{M_2}\lra M_2$ is well-defined and smooth. 
Since $(d\pi_{M_2})_y$ is an orthogonal projection to $T_yM_2$ 
for each $y\in M_2$, 
the map $G_\ve:= \pi_{M_2} \circ F_\ve : M_1 \lra M_2$ 
is an immersion for any sufficiently small $\ve>0$. 
Since $G_\ve (M_1)$ is open and closed in $M_2$, 
$G_\ve (M_1)=M_2$. Moreover, since $M_1$ is compact and $M_2$ 
is Hausdorff, $G_\ve$ is proper, and thus $G_\ve$ is a covering map from $M_1$ onto $M_2$. 
Since $M_2$ is simply connected, $G_\ve$ is injective, 
and hence it is a diffeomorphism 
from $M_1$ onto $M_2$ for any sufficiently small $\ve >0$.$\qedd$

\bigskip\noindent
{\it (Proof of the differentiable exotic sphere theorem II)} 
Let $F:M_1\lra M_2$ and $\wt{F}:\B_\pi (o_{q_1})\lra \B_\pi (o_{q_2})$ 
be the maps defined by \eqref{2015_08_24_new1} and \eqref{2015_08_24_new2}, respectively. By the same argument 
in the proof of the case where \eqref{2015_07_25_MA_1} holds, 
we have $\partial F(q_1)
= \partial \wt{F}(o)
= \Conv (\{A_v \,|\, v \in \Sph^{n-1}_{q_1}\})
$, where $A_v$ is as in Lemma \ref{2015_07_26_lemma7}. 
Since $\sigma$ satisfies \eqref{2016_03_23_TAII_as}, 
by Lemma \ref{2016_03_23_lem1}, 
$q_1$ is non-singular for $F$, i.e., 
$F$ has no singular points on $M_1$. 
By Theorem \ref{2015_07_30_theorem2} and 
the same argument as the proof of the case where \eqref{2015_07_25_MA_1} holds, 
$M_1$ and $M_2$ are diffeomorphic.$\qedd$

\subsection{The reason why we need the restriction of $\ddot{c}$ in the condition \eqref{2015_07_25_MA_1}}\label{sect2_new4}

Let $\sigma: S^{n-1}(1)\lra S^{n-1}(1)$ be a diffeomorphism, 
where $S^{n-1}(1):=\{v \in \R^n\,|\,\|v\|=1\}$, and let 
$\bL(\sigma)\ge1$ denote the bi-Lipschitz constant of $\sigma$ 
defined by \eqref{2015_08_26_MA_3}. First, we will prove that, for any $i,j \in \{1,2,\ldots,n\}$, 
\[
|\langle \sigma(e_i),\sigma(e_j) \rangle -\delta_{ij}|\le \bL(\sigma)^2-1,
\]
where $e_i:= (0, \ldots, 1, \ldots, 0) \in \R^n$. 
From this property, $\{\sigma(e_1),\sigma(e_2), \ldots, \sigma(e_n)\}$ looks  
linearly independent if $\bL(\sigma)^2-1$ is sufficiently small.
As Lemma \ref{2015_07_27_lemma9} below  shows, however, 
$\{\sigma(e_1),\sigma(e_2), \ldots, \sigma(e_n)\}$ 
is not always linearly independent. 
Take any vectors $u,v \in S^{n-1}(1)$. By the parallelogram law, we have 
\[
\langle \sigma(u), \sigma(v) \rangle
= \frac{1}{2}\{\|\sigma(u)\|^2 + \|\sigma(v)\|^2 - \|\sigma(u)- \sigma(v)\|^2\}
= \frac{1}{2}\{2 - \|\sigma(u)- \sigma(v)\|^2\}
\]
and 
$
\langle u,v \rangle
= \{\|u\|^2 + \|v\|^2 - \|u- v\|^2\}/2
= \{2 - \|u-v\|^2\}/2
$.
Then, we have 
\begin{align*}
\langle \sigma(u), \sigma(v) \rangle - \langle u,v \rangle 
&= \frac{1}{2}\{\|u-v\|^2 - \|\sigma(u)- \sigma(v)\|^2\}\\[1mm] 
&\le \frac{1}{2}(1-\bL(\sigma)^{-2})\|u-v\|^2 
\le \frac{\bL(\sigma)^2-1}{2}\|u-v\|^2.
\end{align*}
As well as above, 
$
\langle \sigma(u), \sigma(v) \rangle - \langle u,v \rangle 
\ge (1-\bL(\sigma)^2)\|u-v\|^2/2.
$
Thus, 
\[
|\langle \sigma(u), \sigma(v) \rangle - \langle u,v \rangle| \le \frac{\bL(\sigma)^2-1}{2}\|u-v\|^2
\]
holds. In particular, 
$
|\langle \sigma(e_i), \sigma(e_j) \rangle - \langle e_i,e_j \rangle| 
\le \bL(\sigma)^2-1
$.

\begin{lemma}\label{2015_07_27_lemma9}
For any $\ve >0$, there exist $\mbox{\boldmath $a$}_1, \mbox{\boldmath $a$}_2, 
\ldots, \mbox{\boldmath $a$}_n \in \R^n$ such that 
\begin{equation}\label{2015_07_27_lemma9_1}
|\langle \mbox{\boldmath $a$}_i, \mbox{\boldmath $a$}_j \rangle - \delta_{ij}| < \ve,
\end{equation}
but $\{\mbox{\boldmath $a$}_1, \mbox{\boldmath $a$}_2, 
\ldots, \mbox{\boldmath $a$}_n\}$ is not linearly independent. 
\end{lemma}

\begin{proof}
Take any $\ve \in (0,1)$, and fix it. 
Choose an integer $k \ge 0$ satisfying 
\begin{equation}\label{2015_07_27_lemma9_2}
\frac{4(1-\ve^2)}{\ve^2}< k \le \frac{4(1-\ve^2)}{\ve^2} +1.
\end{equation}
Set $c_1=c_2= \cdots = c_k = \ve/2$. Since $\sum_{j=1}^{k}c_j^2= k\ve^2/4$, we have $1-\sum_{j=1}^{k}c_j^2= 1- k\ve^2/4$. 
By the right inequality of \eqref{2015_07_27_lemma9_2},
\[
1-\sum_{j=1}^{k}c_j^2\ge 1- \frac{\ve^2}{4}\left\{\frac{4(1-\ve^2)}{\ve^2} +1\right\}= \frac{3}{4}\ve^2 >0.
\]
By the left inequality of \eqref{2015_07_27_lemma9_2},
\[
1-\sum_{j=1}^{k}c_j^2 < 1- \frac{\ve^2}{4}\cdot\frac{4(1-\ve^2)}{\ve^2}= \ve^2.
\]
Thus, we have $0< c_{k+1}:= \sqrt{1-\sum_{j=1}^{k}c_j^2} < \ve$. 
Defining 
$\mbox{\boldmath $a$}_1:= e_1, \mbox{\boldmath $a$}_2:=e_2, 
\ldots, \mbox{\boldmath $a$}_k:=e_k, \mbox{\boldmath $a$}_{k+1}
=\sum_{j=1}^{k+1}c_j e_j\in \R^n$, 
we see $\langle \mbox{\boldmath $a$}_i, \mbox{\boldmath $a$}_j\rangle = \delta_{ij}$ 
for all $i,j < k+1$, $\langle \mbox{\boldmath $a$}_{k+1}, \mbox{\boldmath $a$}_{k+1}\rangle =1$, and $0< \langle \mbox{\boldmath $a$}_{k+1}, \mbox{\boldmath $a$}_{i}\rangle <\ve$ for all $i\not=k+1$. Hence, $\{\mbox{\boldmath $a$}_1, \mbox{\boldmath $a$}_2, 
\ldots, \mbox{\boldmath $a$}_{k+1}\}$ is linearly dependent, but that satisfies \eqref{2015_07_27_lemma9_1}.$\qedd$
\end{proof}

\begin{remark}
By Lemma \ref{2015_07_27_lemma9}, 
it looks impossible to find a bound on the Lipschitz constant 
of $\wt{F}$ independent of the dimension $n$ such that the differential 
of the smooth approximation of $\wt{F}$ at $o$ is injective. 
In this sense, it is natural that the Lipschitz constant 
of the locally bi-Lipschitz map in \cite[Theorem 5.1]{K} 
depends on $n$, and it is 
too for \eqref{2015_08_12_MA_2}.
\end{remark}

\section{Proof of the differentiable twisted sphere theorem}\label{sect4}

We need two lemmas in order to prove 
the differentiable twisted sphere theorem. 

\begin{lemma}\label{2014_07_29_lem4.1}
Let $M$ be a complete Riemannian manifold, and let $f:M\lra\R$ be a Lipschitz function on $M$. If $f^{-1}(0)$ is compact, then for any open neighborhood $U$ of $f^{-1}(0)$, 
there exists a positive number $\ve_0$ such that
for all $\ve\in(0,\ve_0)$, $f_\ve^{-1}(0) \subset U$, 
where $f_\ve$ denotes the smooth approximation of $f$ 
defined in Sect.\,\ref{2016_03_12_sect2_sub1}. 
\end{lemma}

\begin{proof}
For each $\ve>0$, 
let $f_\ve$ be the smooth approximation 
of the function $f$, i.e.,
$f_\ve(q):=\sum_i\psi_i(q)f_\ve^{(p_i)}(q)$. 
Here, $\{\psi_i\}_{i=1}^\ell$ denotes the partition 
of unity subordinate to $\{B_{r_i/2} (p_i)\}$, 
where $\{B_{r_i}(p_i)\}$ is a locally finite covering  
of strongly convex balls of $M$ which satisfies 
$M= \cup_i B_{r_i/2}(p_i)$, and the local approximation 
$f_\ve^{(p_i)}$ is defined by the equation \eqref{eq:N1.0}. 
Since $\sum_{i=1}^\ell \psi_i (q) = 1$, 
we have, by the triangle inequality, 
\begin{equation}\label{2014_05_27_thm5.8_proof1}
|f_\ve (q) - f(q)|
= \left| \sum_{i=1}^\ell \psi_i (q) (f_\ve^{(p_i)} (q) -f(q)) \right| 
\le \sum_{i=1}^\ell \psi_i (q) \left| f_\ve^{(p_i)} (q) -f(q) \right|.
\end{equation}
Applying Lemma \ref{lemN5} to \eqref{2014_05_27_thm5.8_proof1}, 
we see that 
\[
|f_\ve (q) - f(q)|
\le \ve \cdot \L(f) \sum_{i=1}^\ell \psi_i (q) 
\L(\exp_{p_i}|_{\B_{r_i} (o_{p_i})})\]
for all $\ve\in(0,\ve_0)$, where $\ve_0:=\min\{\ve_i\,|\,i=1,2, \ldots,\ell\}$. 
Hence, for any sufficiently small $\ve>0$, 
$f^{-1}_\ve(0)$ is a subset of $U$.$\qedd$
\end{proof}

\begin{lemma}\label{2014_07_29_lem4.2}
Let $A$, $B$ be linear transformations on $\R^n$ 
such that $A|_{\R^{n-1}} = B|_{\R^{n-1}} = \id_{\R^{n-1}}$, and that $\langle A\ora{n}, \ora{n}\rangle >0$, $\langle B\ora{n}, \ora{n}\rangle >0$, where $\R^{n-1}:=\{(x_1, x_2, \ldots, x_n) \in \R^n\,|\,x_n=0\}$ and 
$\ora{n}:= (0, \ldots, 0, 1)$. Then, every element in $\Conv (\{A, B\})$ is of maximal rank. 
\end{lemma}

\begin{proof}
Take any $\lambda \in [0,1]$. 
Assume that there exists $\ora{v} +a\ora{n}\in{\R^{n}}$, 
where $a \in \R$ and $\ora{v} \in \R^{n-1}$, 
such that 
$(\lambda A +(1-\lambda) B) (\ora{v} + a\ora{n}) =o$. 
Since $o = (\lambda A +(1-\lambda) B) (\ora{v} + a \ora{n}) 
= \ora{v} + a (\lambda A +(1-\lambda) B) \ora{n}$, we have 
\[
0= \langle  \ora{v} + a (\lambda A +(1-\lambda) B) \ora{n}, \ora{n} \rangle 
= a(\lambda \langle A \ora{n}, \ora{n} \rangle + (1-\lambda) \langle B \ora{n}, \ora{n}\rangle), 
\]
and hence $a=0$. Since $(\lambda A +(1-\lambda) B) (\ora{v} + a\ora{n}) =o$, 
we have $\ora{v} =o$. Thus, $\ora{v} + a \ora{n} =o$. This implies 
that
$\lambda A +(1-\lambda) B$ is non-singular for any $\lambda\in[0,1]$.$\qedd$
\end{proof}

\medskip

In what follows, all notations are the same as those defined 
in Sect.\,\ref{2016_01_11_new}.

\medskip\noindent
{\it (Proof of the differentiable twisted sphere theorem)} 
We first prove (T-1), i.e., $M$ is a twisted sphere: 
Let $d_p, d_q$ be the distance functions from $p, q$, 
respectively, i.e., $d_p(x):=d(p,x)$ for all $x \in M$. 
Consider the Lipschitz function 
$f := d_p - d_q$ on $M$. Remark that $f^{-1}(0) = E_{p,\,q}$ is compact. 
Let  $f_\ve$ denote the smooth approximation of $f$, 
i.e., $f_\ve= (d_p)_\ve - (d_q)_\ve$. 
Applying Lemma \ref{lemN12} to the compact set $K=E_{p\,,q}$,
we obtain that $\nabla f_\ve\ne 0$ on $E_{p\,,q}$ 
for all sufficiently small $\ve>0$. 
Choose any $R>0$ so as to be $p, q \not\in B_R(f^{-1}(0))$. 
Here $B_R(f^{-1}(0)):=\{x\in M\,|\,\wt{d}(f^{-1}(0), x)<R\}$, 
where $\wt{d}(f^{-1}(0), x)= \min_{y\in f^{-1}(0)} d(y,x)$. 
By Lemma \ref{2014_07_29_lem4.1}, 
$f_\ve^{-1}(0) \subset B_R(f^{-1}(0))$ 
for any sufficiently small $\ve>0$, 
and hence $f_\ve^{-1}(0)$ is a regular compact 
hypersurface. 
Choose such a sufficiently small $\ve$, and fix it in the following.
Now, let $\ol{D_\ve(p)}:=\{ x\in M\,| \, f_\ve(x)\leq 0\}$ 
and $\ol{D_\ve(q)}:=\{ x\in M\,| \, f_\ve(x)\geq 0\}$. 
Since $d_{p} $ is smooth on a punctured convex ball $\cP (p)$ at $p$, 
$\| \nabla (d_p)_\ve-\nabla d_p \|$ is sufficiently small on $\cP (p)$. Therefore, we can assume that there exists a non-zero smooth vector field $X_+$ on $\ol{D_\ve(p)}\setminus\{p\}$ such that
\[
X_+ = 
\begin{cases}
\ \nabla d_p \ &\text{on} \ \cP (p),\\[1mm]
\ \nabla (d_p)_\ve  \ &\text{on a neighborhood of} \ f_\ve^{-1}(0).
\end{cases}
\]
For each $v \in \Sph^{n-1}_{p}$, 
let $\tau_+ (t, v)$ be the integral curve of $X_+$ 
with initial conditions $\tau_+ (0, v)= p$ and 
$\frac{\partial \tau_+}{\partial t} (0, v) = v$.  
Since $X_+ = \nabla d_{p}$ on $\cP (p)$, 
$\tau_+ (t, v) = \exp_{p} tv$ holds 
for all sufficiently small $t\ge0$. 
It follows from Lemma \ref{lemN12} that 
$\angle (\nabla (d_p)_\ve, \nabla (d_q)_\ve) > \pi/2$ on $E_{p\,,q}$. 
Thus, there exists a smooth solution $t=t_+(v) > 0$ of 
$f_\ve (\tau_+(t, v)) =0$, 
since $X_+$ is not tangent to $f_\ve^{-1} (0)$ at each point of the hypersurface. 
Then, we have the diffeomorphism 
$G_+$ from $\ol{D_\ve(p)}$ onto $S^n_+(1)
:=\{(x_1,x_2,\ldots,x_{n+1})\in S^n(1)\,|\, x_{n+1}\ge  0\}$ 
defined by 
\[
G_+(\tau_+(t,v)):=\exp_{N} \frac{t\pi}{2t_+(v)}I(v)
\] 
for all $(t,v)\in[0,\infty)\times \Sph^{n-1}_{p}$ with $t\leq t_+(v)$, 
where $I$ denotes a linear isometry from $T_pM$ onto $T_{N}S^n(1)$.
In the same way as $G_+$, we have the diffeomorphism 
$G_-$ from $\ol{D_\ve(q)}$ onto $S^n_-(1)
:=\{(x_1,x_2,\ldots,x_{n+1})\in S^n(1)\,|\, x_{n+1}\le 0\}$. 
Thus, we can define the induced metrics $g_{\pm}:=G^*_{\pm}g_0$
on $\ol{D_\ve(p)}, \ol{D_\ve(q)}$ from $S^n(1)$, 
where $g_0$ denotes the metric of $S^n(1)$, 
so that $\ol{D_\ve(p)}, \ol{D_\ve(q)}$ are isometric 
to $S^n_+(1), S^n_-(1)$, respectively. 
Let $\exp^{g_+}$, $\exp^{g_-}$ be exponential maps 
on tangent spaces at $p$, $q$ of $\ol{D_\ve(p)}$, 
$\ol{D_\ve(q)}$ with respect to $g_\pm$, respectively. 
For each $(t, v) \in [0, \pi] \times \Sph_{p}^{n-1}$, 
we define the point $e(t, v)$ on $M$ with respect 
to $g_\pm$ by 
\[
e (t, v) = 
\begin{cases}
\ \exp^{g_+} (t v) \ &\text{on} \ [0,\pi/2],\\[1mm]
\ \exp^{g_-}( (\pi - t)\sigma^p_q(v))  \ &\text{on} \ [\pi/2,\pi],
\end{cases}
\]
where $\sigma^p_q: \Sph_p^{n-1}\lra\Sph_q^{n-1}$ is the diffeomorphism 
satisfying $\exp^+(\pi v/2 )=\exp^-(\pi \sigma^p_q (v)/2)$. 
Thus, we have a boundary diffeomorphism 
$h_{\sigma^p_q}:\partial S^n_+(1)\lra \partial S^n_-(1)$ induced 
from $\sigma^p_q$. Hence, 
$M= S^n_+(1)\cup_{h_{\sigma^p_q}}S^n_-(1)$ 
is twisted.\par
We next prove (T-2), i.e., we construct a bi-Lipschitz homeomorphism 
from $M$ to $S^n(1)$ that admits a diffeomorphism 
between $M \setminus \{q\}$ and $S^n(1) \setminus \{S\}$: 
For every $(t,v)\in[0,\pi]\times\Sph_p^{n-1}$, define 
the map $F:M\lra S^n(1)$ by 
\[
F(e(t,v)):=\exp_N tI(v). 
\]
It is not difficult to see that $F$ is bi-Lipschitz.
Since $F$ is a local diffeomorphism on 
$M \setminus (f^{-1}_\ve (0)\cup \{q\})$, 
$F$ has no singular points on $M \setminus (f^{-1}_\ve (0)\cup \{q\})$. 
Let $F^+$ and $F^- $ 
be the smooth extensions of $F|_{\ol{D_\ve(p)}}$ 
and $F|_{\ol{D_\ve(q)}\setminus\{q\}}$, respectively. 
Since $\partial F(x)= \Conv(\{dF^+_x, dF^-_x\})$ 
for each $x \in f^{-1}_\ve (0)$, 
it follows from Lemma \ref{2014_07_29_lem4.2} that 
any element in $\partial F(x)$ is of maximal rank. 
Hence, $F$ has no singular points on $M \setminus \{q\}$. 
Since $B_R(f^{-1}(0))$ is open in $M_1$ 
and $f^{-1}_\ve(0) \subset B_R(f^{-1}(0))$, 
there exists $r\in (0,R)$ such that 
$B_r(f^{-1}_\ve(0)) \subset B_R(f^{-1}(0))$. 
Note that $p, q \not\in B_r(f^{-1}_\ve(0))$. 
Let $\varphi$ be a smooth function on $M$ satisfying 
$0\le\varphi \le1$ on $M$, 
$\varphi\equiv1$ on $\ol{B_r(f^{-1}_\ve(0))}$, and 
$\supp \varphi \subset B_R(f^{-1}(0))$. 
Define the map 
$G_\ve:M\lra \R^{n+1}$ 
by $G_\ve :=(1-\varphi)F +\varphi F_\ve$, 
where $F_\ve:M\lra \R^{n+1}$ 
denotes the smooth approximation of $F$ defined by \eqref{eq:B2}. 
It is clear that 
\[
G_\ve = 
\begin{cases}
\ F_\ve \ &\text{on} \ \ol{B_r(f^{-1}_\ve(0))},\\[1mm]
\ F  \ &\text{on} \ M\setminus \supp \varphi.
\end{cases}
\]
Since $F$ has no singular points 
on $\ol{B_r(f^{-1}_\ve(0))}$, 
by the same argument as the proof 
of Lemma \ref{lemB11}, $dF_\ve$ is injective  
on $B_r(f^{-1}_\ve(0))$. Since $\ve>0$ is sufficiently small, 
by \eqref{2014_05_27_thm5.8_proof1}, 
we can assume that $|f_\ve-f|\le r$ on $M$. 
Then, since $F$ is smooth 
on $B_R(f^{-1}(0)) 
\setminus \ol{B_r(f^{-1}_\ve(0))}$, 
$F_\ve$ uniformly 
converges to $F$ on $\ol{B_R(f^{-1}(0))} 
\setminus B_r(f^{-1}_\ve(0))$ 
as $\ve\downarrow0$ in the $C^1$-topology. 
Thus, we see that 
$G_\ve$ uniformly 
converges to $F$ on $\ol{B_R(f^{-1}(0))} 
\setminus B_r(f^{-1}_\ve(0))$ 
as $\ve\downarrow0$ in the $C^1$-topology.
Hence, $G_\ve:M\lra\R^{n+1}$ is a bi-Lipschitz homeomorphism 
which is a local diffeomorphism on $M\setminus\{q\}$. 
Define the map $\psi_\ve:M\lra S^n(1)$ 
by $\psi_\ve:=\pi_{S^n(1)}\circ G_\ve$, 
where $\pi_{S^n(1)}:\R^{n+1}\setminus\{o\}\lra S^n(1)$ 
denotes the distance projection. 
By a similar argument to the proof of the differentiable 
exotic sphere theorem, we see that $\psi_\ve$ 
is a covering map from $M$ 
onto $S^n(1)$. Since $S^n(1)$ is simply connected, 
$\psi_\ve$ is injective. Therefore, 
$\psi_\ve$ is a bi-Lipschitz homeomorphism from $M$ onto $S^n(1)$ 
which is a diffeomorphism except for $q$.\par 
Finally, we prove (T-3): 
Define the map $\wt{F}: \B_\pi(o_q)\lra T_S S^n(1)$ by 
$\wt{F}:= \exp_S^{-1}\circ F \circ \exp^{g_-}$. 
By a similar argument to the proof of the differentiable 
exotic sphere theorem, we then see 
that $\wt{F}(tv)= t \sigma^N_S\circ I \circ (\sigma^{p}_{q})^{-1}(v)$ 
for all $(t, v) \in [0, \pi] \times \Sph_{q}^{n-1}$. 
Here, $\sigma^{N}_{S}$ denotes the diffeomorphism defined 
by \eqref{2015_07_30_def_of_subsigma}. 
In particular, 
$\wt{F}$ satisfies \eqref{2015_07_26_lemma7_1}. Hence, 
by Lemma \ref{2015_07_26_lemma7}, 
$\partial F(q)= \Conv (\{A_v \,|\, v \in \Sph^{n-1}_{q}\})$, 
where $A_v$ is as in the lemma. 
For a geodesic segment $\gamma:[0,\pi]\lra \Sph^{n-1}_q$, 
let $c:[0,\pi]\lra \Sph^{n-1}_{S}$ be the curve defined 
by 
$c:= \sigma^N_S\circ I \circ (\sigma^{p}_{q})^{-1} \circ \gamma$.  Now, we assume $\bL((\sigma^{p}_{q})^{-1})$ satisfies \eqref{2015_07_25_MA_1} 
for all geodesic segments 
$\gamma ([0,\pi])\subset \Sph^{n-1}_q$ with $\|\dot{\gamma}\|\equiv 1$. 
Since 
\[
\bL(\sigma^N_S\circ I \circ (\sigma^{p}_{q})^{-1})
=\bL((\sigma^{p}_{q})^{-1}),
\]
$\bL(\sigma^N_S\circ I \circ (\sigma^{p}_{q})^{-1})$ 
satisfies \eqref{2015_07_25_MA_1}. 
By Corollary \ref{2015_08_31_cor}, 
$q$ is non-singular for $F$, 
and hence $F$ has no singular points on $M$. 
Therefore, it follows from Theorem \ref{2015_07_30_theorem2} 
that $M$ and $S^n(1)$ are diffeomorphic. 
In the case where \eqref{2015_08_12_MA_2} or \eqref{2016_03_23_TAII_as} holds, 
we can prove that $M$ and $S^n(1)$ are diffeomorphic 
by the same arguments as Sect.\,\ref{sect2_new3}.
$\qedd$

\begin{remark}\label{2015_08_08_rem} 
By applying \cite[Proposition C]{W} to the two discs 
$\ol{D_\ve(p)}$ and $\ol{D_\ve(q)}$ in the proof above successively, 
we get a new Riemannian metric on $M,$ points $p_1\in D_\ve (p)$ 
and $q_1\in D_\ve (q)$ such that all geodesics (with this new metric) 
emanating from $p_1$ (respectively $q_1$) pass through the boundary of $\ol{D_\ve (p)}$ 
(respectively $\ol{D_\ve (q)}$) perpendicularly. Therefore, 
any geodesic emanating from $p_1$ passes through $q_1$, i.e., the cut locus of $p_1$ consists of $q_1$.
\end{remark}

\begin{acknowledgement}
The authors express their sincere thanks to Professors K. Grove, P. Petersen, and F. Wilhelm, who  pointed out a mistake in the first version, back then entitled
``A sufficient condition for a pair of bi-Lipschitz
homeomorphic manifolds to be diffeomorphic and sphere theorems'', of this article. 
Thanks to the indication, they could 
arrive at the differentiable exotic sphere theorem. The first named author would like to thank Professor T. Shioya, who suggested the 
shorter proof of Lemma \ref{2014_02_05_lem1} than that in the first version. He is deeply grateful to Professor F.H. Clarke for his encouragement, to Professors 
M. Gromov, S. Ohta, 
and T. Yamaguchi for their comments on the first one, 
and to Professors Y. Agaoka, M. Ishida, 
and K. Yasui for their having given knowledge 
about exotic structures, which has had a big 
influence on the structure of Sect.\,\ref{sec1}. 
Finally, they thank the referees cordially for careful reading of the manuscript and for helpful and valuable comments on it, 
which have improved the presentation of this article.
\end{acknowledgement}

\medskip

\begin{flushleft}
K.~Kondo\\
Department of Mathematical Sciences, Yamaguchi University\\
Yamaguchi City, Yamaguchi Pref. 753-8512, Japan\\
{\small e-mail: {\tt keikondo@yamaguchi-u.ac.jp}}

\medskip

M.~Tanaka\\
Department of Mathematics, Tokai University\\
Hiratsuka City, Kanagawa Pref. 259-1292, Japan\\
{\small e-mail: {\tt tanaka@tokai-u.jp}}
\end{flushleft}

\end{document}